\documentclass[12pt,preprint]{amsart}

\usepackage{amsmath, latexsym, amsfonts, amssymb, amsthm, mathrsfs}
\usepackage{graphicx, color, dsfont, epsfig, caption, wrapfig, subfig, bm}
\usepackage{pstricks, yfonts, mathalfa}
\usepackage{hyperref}
\usepackage{mathtools}
\usepackage{parskip}
\usepackage{epsfig}
\usepackage{soul}

\hypersetup{
	linktoc=page,
	linkcolor=red,          
	citecolor=blue,        
	filecolor=blue,      
	urlcolor=cyan,
	colorlinks=true  
	hidelinks=true           
}

\theoremstyle{definition}
\newtheorem{defdef}{Definition}[section]
\newtheorem{defi}{Definition}[section]
\theoremstyle{plain}
\newtheorem{thm}[defi]{Theorem}
\newtheorem{lemma}[defi]{Lemma}
\newtheorem{prop}[defi]{Proposition}
\newtheorem{coro}[defi]{Corollary}

\newtheorem{rem}[defi]{Remark}

\numberwithin{equation}{section}

\newcommand{\C}{\mathbb{C}}
\newcommand{\R}{\mathbb{R}}
\newcommand{\N}{\mathbb{N}}

\newcommand{\E}{\mathbb{E}} 

\renewcommand{\S}{\mathbb{S}}
\renewcommand{\H}{\mathbb{H}} 

\renewcommand{\P}{\mathbb{P}}

\def\A{\bm{\mathrm{A}}}
\def\vac{\vert 0\rangle}

\renewcommand{\c}{\mathbf{c}}

\newcommand{\g}{\mathbf{g}}
\renewcommand{\k}{\mathbf{k}}

\newcommand{\bphi}{\bm{\phi}}

\newcommand{\X}{\mathbf{X}}

\newcommand{\Z}{\mathbf{Z}}
\newcommand{\V}{\bm{V}}

\newcommand{\mc}[1]{\mathcal{#1}}
\newcommand{\MW}[1]{\prescript{#1}{}{\mc M_+}}

\renewcommand{\d}{\text{\rm d}}

\newcommand{\ps}[1]{\langle #1 \rangle}

\newcommand{\ind}{\mathds{1}}

\newcommand{\Id}{\text{\rm{Id}}}

\def\Wb{\bm{\mathrm W}}

\def\Db{\bm{\mathrm D}}

\renewcommand{\Re}{\text{\rm{Re}}}
\renewcommand{\Im}{\text{\rm{Im}}}

\newcommand{\neu}{\partial_{n_g}}

\renewcommand{\bar}[1]{\overline{#1}}

\newcommand{\eps}{\varepsilon}

\renewcommand{\g}{\mathfrak g}
\renewcommand{\a}{\mathfrak a}
\renewcommand{\sl}{\mathfrak{sl}}

\newcommand{\norm}[1]{\left\lvert#1\right\rvert}
\newcommand{\expect}[1]{\mathbb{E}\left[#1\right]}
\newcommand{\eqlaw}{\overset{\text{(law)}}{=}}
\newcommand{\qt}[1]{\quad\text{#1}\quad}
\NewDocumentEnvironment{eqs}{+b}
{\begin{equation}\begin{split}#1\end{split}\end{equation}}
{}
\usepackage{bm}
\usepackage{bbm}

\def\adots{\mathinner{\mkern2mu\raise 1pt\hbox{.}\mkern 3mu\raise 
		4pt\hbox{.}\mkern1mu\raise 7pt\hbox{{.}}}}

\title[Boundary Toda CFT]{Boundary Toda Conformal Field Theory from the path integral}
\author[B. Cercl\'e]{Baptiste Cercl\'e}
\email{baptiste.cercle@epfl.ch}
\address{EPFL SB MATH RGM, MA B2 397, Station 8, CH-1015 Lausanne, Switzerland.}
\author[N. Huguenin]{Nathan Huguenin}
\email{nathan.huguenin@univ-amu.fr}
\address{Aix-Marseille Université, CNRS, Institut de Mathématiques de Marseille (I2M) – UMR 7373, Site de Saint Charles, 3 place Victor Hugo, Case 19, 13331 Marseille cédex 3, France.}
\begin{document}
	\maketitle
	
	\begin{abstract}
		Toda Conformal Field Theories (CFTs hereafter) are generalizations of Liouville CFT where the underlying field is no longer scalar but takes values in a finite-dimensional vector space induced by a complex simple Lie algebra. The goal of this document is to provide a probabilistic construction of such models on all compact hyperbolic Riemann surfaces with or without boundary. To do so, we rely on a probabilistic framework based on Gaussian Free Fields and Gaussian Multiplicative Chaos.
		
		In the presence of a boundary, one major difference with Liouville CFT is the existence of non-trivial outer automorphisms of the underlying Lie algebra. The main novelty of our construction is to associate to such an outer automorphism a type of boundary conditions for the field of the theory, leading to the definition of two different classes of models, with either Neumann or Cardy boundary conditions.  This in turn has implications on the algebra of symmetry of the models, which are given by $W$-algebras.
	\end{abstract}
	
	
	\section{Introduction}\label{sec:intro}
	\subsection{Toda Conformal Field Theories}
	
	\subsubsection{Liouville theory on closed surfaces}
	The probabilistic framework introduced by David-Guillarmou-Kupiainen-Rhodes-Vargas~\cite{DKRV16, DRV16, GRV19} to provide a mathematical study of Liouville Conformal Field Theory (CFT hereafter) has led to major breakthroughs in the understanding of this model of two-dimensional random geometry. Among other achievements are the probabilistic derivation of the celebrated DOZZ formula~\cite{DO94, ZZ96} in~\cite{KRV_DOZZ} and a mathematically rigorous implementation of the conformal bootstrap procedure~\cite{GKRV} as well as Segal's axioms~\cite{Seg04, GKRV_Segal}.  In this respect the probabilistic framework has proved to be key in the mathematical formulation and derivation of predictions from the physics literature.
	
	Indeed, before becoming a mathematical topic of research Liouville CFT was initially introduced in the physics literature in the groundbreaking work of Polyakov~\cite{Pol81} where a notion of two-dimensional random geometry, natural in the setting of string theory and quantum gravity, was designed. It has then become a key object of study in the understanding of two-dimensional CFT following the introduction of the \textit{conformal bootstrap method} by Belavin-Polyakov-Zamolodchikov~\cite{BPZ} (BPZ hereafter) since (among other things) it provides a canonical model in which this procedure can be implemented. In the language of physics Liouville CFT is defined using a path integral, which means in the present context that one defines a measure on a functional space $\mc F\subset L^2(\Sigma,g)$ with $(\Sigma,g)$ a smooth, compact Riemannian surface. This measure is formally defined so that the average of a functional $F$ over $\mc F$ takes the form 
	\begin{equation}\label{eq:path_integral}
		\langle F(\Phi) \rangle_{g} \coloneqq \frac1{\mc Z}\int_{\bphi\in\mc F} F( \bphi)e^{-S_L(\bphi,g)}D \bphi
	\end{equation}
	where $\mathcal{Z}$ is a normalization factor and $D \bphi$ stands for a \lq\lq uniform measure" over $\mc F$, while $S_L:\mc F\to\R$ is the so-called \textit{Liouville action}. When $\Sigma$ has no boundary the latter is
	\begin{equation}\label{eq:Toda_action}
		S_{L}(\bphi,g)\coloneqq  \frac{1}{4\pi} \int_{\Sigma}  \Big (  \norm{d_g\bphi}^2   +Q R_g\bphi +4\pi \mu e^{\gamma    \bphi}   \Big)\,{\rm d}v_{g}
	\end{equation}
	where $R_g$ is the scalar curvature, $d_g$ the gradient and $v_g$ the volume form, all taken in the metric $g$. The action also features a coupling constant $\gamma\in(0,2)$ and the background charge $Q\coloneqq \frac\gamma2+\frac2\gamma$, as well as a cosmological constant $\mu\geq0$. Based on this heuristic, the correlation functions of Liouville CFT are given by considering special observables $F$ in Equation~\eqref{eq:path_integral}, given by products of \textit{Vertex Operators} $V_{\alpha}(z)=e^{\alpha\Phi(z)}$ where $\alpha\in\R$ is a \textit{weight} and $z\in\Sigma$ an \textit{insertion}. Hence, the correlation functions are the quantities formally defined by
	\begin{equation}\label{eq:formal_correl}
		\langle \prod_{k=1}^NV_{\alpha_k}(z_k)\rangle_{g} \coloneqq \frac1{\mc Z}\int_{\bphi\in\mc F} \prod_{k=1}^Ne^{\alpha_k\bphi(z_k)}e^{-S_L( \bphi)}D \bphi.
	\end{equation}
	
	However this approach used in the physics does not really make sense mathematically speaking; in this perspective one of the key achievements of David-Kupiainen-Rhodes-Vargas~\cite{DKRV16} has been to provide a rigorous probabilistic framework to give a meaning to the above, hence translating this heuristic into a well-defined object of study.
	
	\subsubsection{Toda CFTs on closed Riemann surfaces}
	The conformal bootstrap method envisioned by BPZ~\cite{BPZ} has been successfully implemented in the physics literature, among other examples, for a large family of models of statistical physics at criticality: this is the case for instance for \textit{minimal models} such as the planar Ising model at criticality. However, there is a whole class of CFTs enjoying in addition to conformal symmetry an enhanced level of symmetry: the critical three-states Potts model is a prototypical example of such a phenomenon in that it admits a \textit{higher-spin symmetry}. In order to provide a suitable generalization of the conformal bootstrap procedure to the setting of models enjoying such a higher level of symmetry, Zamolodchikov introduced in 1985~\cite{Za85} the notion of $W$-algebra, which allowed him to extend the method of BPZ to models featuring \textit{higher-spin} or \textit{W}-symmetries.
	
	In this perspective the algebra of symmetry of Liouville CFT, the Virasoro algebra, corresponds to the case where the $W$-algebra is associated to the simple and complex Lie algebra $\mathfrak{sl}_2$. Toda CFTs allow one to extend this construction by generalizing Liouville theory to a family of CFTs indexed by simple and complex Lie algebras $\mathfrak{g}$ and for which the algebras of symmetry are given by simple $W$-algebras. Despite playing a role similar to Liouville CFT in this generalized setting and therefore having being studied in depth in the physics literature, due to their complexity many features of Toda CFTs remain to be understood.
	
	In contrast with Liouville theory, the field in Toda CFTs has values in a $r$-dimensional vector space, where $r$ stands for the rank of the underlying Lie algebra. The problem then becomes the one of defining a measure $\mu_T$ on a functional space $\mc F\subset L^2(\Sigma\to\a,g)$ where $\a\simeq\R^r$. This measure on $\mc F$ is formally defined by $\mu_T(D\bphi) \coloneqq e^{-S_T(\bphi,g)} D\bphi$
	where $D\phi$ stands for the hypothetical uniform measure on this space of functions. Put differently we wish to define for $F$ a suitable functional over $\mc F$
	\begin{equation} \label{eq:toda path integral}
		\ps{F(\Phi)}_{g} \coloneqq \frac1{\mathcal{Z}} \int_{\bphi\in\mc F} F(\bphi) e^{-S_T(\bphi,g)} D\bphi.
	\end{equation}
	The Toda action is given, when $\Sigma$ has no boundary, by the following Lagrangian:
	\begin{equation} \label{eq:toda_action}
		\begin{split}
			S_T(\bphi,g) = \frac{1}{4\pi} \int_\Sigma \left(|d_g \bphi|_g^2 + R_g \ps{Q,\bphi} + \sum_{i=1}^{r} 4\pi\mu_i e^{\ps{\gamma e_i,\bphi}} \right) \d v_g,\qt{where}
		\end{split}
	\end{equation}
	\begin{itemize}
		\item $\gamma \in (0,\sqrt{2})$\footnote{This range of values for the coupling constant differs from the usual one in that the simple roots are normalized so that the longest ones have norm $\sqrt2$. To recover the standard convention one therefore needs to rescale $\gamma$ by $\sqrt2$.} is the coupling constant and $Q=\gamma\rho+\frac2\gamma\rho^\vee\in\a$ (see Subsection~\ref{subsec:Lie}) is the background charge;
		\item the $\mu_i$, $1\leq i\leq r$, are the cosmological constants;
		\item the $(e_i)_{1\leq i\leq r}$ are the \textit{simple roots} of $\g$; they form a basis of $\a$. Likewise $\ps{\cdot,\cdot}$ is a scalar product inherited from the Lie algebra structure.
	\end{itemize}
	
	A probabilistic definition for quantities given by Equation~\eqref{eq:toda path integral} has been proposed in~\cite{CRV21} when the underlying Riemann surface is the two-dimensional sphere $\S^2$. To do so, the authors generalized the framework designed in~\cite{DKRV16} in the context of Toda CFTs. This in particular allowed to recover predictions made in the physics literature concerning the symmetries of such models in~\cite{Toda_OPEWV}, a key ingredient in the derivation~\cite{Toda_correl1, Toda_correl2} of the Fateev-Litvinov formula~\cite{FaLi1} for a family of structure constants of the theory. In this document we extend the construction proposed there by making sense of the path integral~\eqref{eq:toda path integral} in several different cases, and first for hyperbolic Riemann surfaces without boundary. 
	
	\subsubsection{Boundary Toda CFT}
	But the main contribution of the present document is to initiate the mathematical study of Toda CFTs on surfaces with a non-empty boundary, which makes apparent some key notions in the study of CFT in the presence of a boundary. Explicitly to any outer automorphism of the underlying Lie algebra we associate a type of boundary conditions for the field of the theory, giving rise to a probabilistic model that features some invariance property with respect to this automorphism. This invariance manifests itself at the level of the algebra of symmetry. Let us explain this in more details.
	
	Boundary Toda CFTs are defined using a path integral in the same fashion as in~\eqref{eq:toda path integral}, where in the presence of a boundary the action functional needs to be modified as follows:
	\begin{equation} \label{eq:toda_action_boundary}
		\begin{split}
			S_T(\bphi,g) = \frac{1}{4\pi} \int_\Sigma \left(|d_g \bphi|_g^2 + R_g \ps{Q,\bphi} + \sum_{i=1}^{r} 4\pi \mu_{B,i} e^{\ps{\gamma e_i,\bphi}} \right) \d v_g \\ + \frac{1}{2\pi} \int_{\partial\Sigma} \left( k_g \ps{Q,\bphi} + \sum_{i=1}^{r} 2\pi\mu_i e^{\ps{\frac{\gamma e_i}{2},\bphi}}  \right)\d \lambda_g
		\end{split}
	\end{equation}
	where $\lambda_g$ is the line element on $\partial\Sigma$ and $k_g$ the geodesic curvature, while the $\mu_i\in\C$, $i=1,\cdots,r$ are boundary cosmological constants. In principle one could directly extend the construction from the closed case to the open one; however, in the presence of a boundary one needs to prescribe boundary conditions for the field $\Phi$ of the theory. This is a key feature of Toda CFTs compared to the Liouville one, for which the target Lie algebra ($\sl_2$) does not admit any non-trivial outer automorphism. Let us explain this in the simplest case where the underlying surface is the upper-half plane $\H$ equipped with the metric $g=\frac{\norm{dz}^2}{(1+\norm{z}^
		2)^2}$.
	
	Based on the so-called \textit{doubling trick}, we may want to realize the Toda field $\Phi:\H\to\a$ using a field $\hat\Phi:\C\to\a$ that lives over the \textit{full} complex plane. To do so we can set
	\begin{equation}\label{eq:Phi_intro}
		\Phi(z)\coloneqq \frac{\hat\Phi(z)+\tau \left(\hat\Phi(\bar z)\right)}{\sqrt 2}
	\end{equation}
	where $\tau$ is a linear involution of $\a$. Then if $\tau=\Id$, $\Phi$ has Neumann boundary conditions, while if $\tau=-\Id$, it has Dirichlet ones. More generally $\Phi$ can be decomposed as a sum of two fields with respectively Neumann and Dirichlet boundary conditions. Now any outer automorphism $\tau$ of $\g$ which has order $1$ or $2$ (which is always the case except when $\g=D_4$) satisfies this assumption. As a consequence, with such a $\tau$ is naturally associated a type of boundary conditions for the field of the theory, thanks to this \textit{doubling trick}. 
	
	In view of the above discussion we propose the following probabilistic definition for boundary Toda CFTs. To start with fix $\tau$ and write the orthogonal decomposition $\a=\a_N\oplus\a_D$ with $\a_N$ the maximal subspace of $\a$ such that $\tau_{\vert\a_N}=\Id$. Then, for $g$ a uniform metric of type 1 (see Subsection~\ref{subsec:uniform}), we let $\X^{g,C}=\X^{g,N}+\X^{g,D}$ where $\X^{g,N}:\Sigma\to \a_{N}$ is a Gaussian Free Field (GFF) with Neumann boundary conditions and $\X^{g,D}:\Sigma\to \a_{D}$ is a Dirichlet one.  The correlation functions of the model, which are functions of insertions $z_1,\cdots,z_N$ in $\Sigma$ and $s_1,\cdots,s_l$ on $\partial\Sigma$ as well as weights $\alpha_1,\cdots,\alpha_N,\beta_1,\cdots,\beta_M$ in $\a$, are then defined by a limiting procedure:
	\begin{equation}\label{eq:def_correl_intro}
		\begin{split}
			&\ps{F(\Phi)\prod_{k=1}^NV_{\alpha_k}(z_k)\prod_{l=1}^MV_{\beta_l}(s_l)}_{g,\tau}\coloneqq \lim\limits_{\eps\to0}\frac1{\mc Z}\int_{\a_N}e^{-\ps{Q\chi(\Sigma),\c}}   \E\Big[F\left(\X^{g,C}+\c\right)\prod_{k=1}^N\eps^{\frac{\norm{\alpha_k}^2}{2}}e^{\ps{\alpha_k,\X^{g,C}_\eps(z_k)+\c}}\\
			&\times\prod_{l=1}^M\eps^{\frac{\norm{\beta_l}^2}{4}}e^{\ps{\frac12\beta_l,\X^{g,C}_\eps(s_l)+\c}}\exp\Big(-\sum_{i=1}^r\mu_{B,i}\int_{\Sigma}e^{\ps{\gamma e_i,\X^{g,C}+\c}}\d v_g-\sum_{j=1}^{d_N}\mu_j\int_{\partial\Sigma}e^{\ps{\frac\gamma2 f_j,\X^{g,C}+\c}}\d\lambda_g\Big)\Big]\d\c.
		\end{split}
	\end{equation}
	for $F$ continuous and bounded on a suitable functional space, where $e^{\ps{\gamma e_i,\X^{g,C}}}\d v_g$ is a Gaussian Multiplicative Chaos (GMC) measure, $d_N=\dim(\a_N)$ and the $(f_j)_{1\leq j\leq d_N}$ form a basis of $\a_N$ (see Subsection~\ref{subsec:fold}). More details on this construction are given in Sections~\ref{sec:closed} and~\ref{sec:bdry}: we extend there the construction to $g$ smooth. Let us mention that the mathematical study of boundary Liouville CFT was initiated in~\cite{HRV18, remy_annulus, Wu}, paving the way to the computation of several key correlation functions of the theory, starting from the derivation of the so-called Fyodorov-Bouchaud formula~\cite{remy1} to the computation of all structure constants in~\cite{ARSZ} building on BPZ-type differential equations~\cite{Ang_zipper}, recovering formulas predicted in the physics literature~\cite{FZZ, PT02, Hos}.

	\subsection{Statements of the main results and perspectives}
	Let $(\Sigma,g)$ be a smooth, connected, compact Riemannian surface with (possibly empty) boundary $\partial\Sigma$, and assume that $\chi(\Sigma)<0$. We also fix $\g$ a complex simple Lie algebra as well as $\tau$ an automorphism of the associated root system $\mc R$, with $\tau=\Id$ if $\partial\Sigma=\emptyset$. We further consider $z_1,\cdots,z_N$ in $\Sigma$ and $s_1,\cdots,s_l$ on $\partial\Sigma$, all distinct, and $\alpha_1,\cdots,\alpha_N,\beta_1,\cdots,\beta_M$ in $\a$.
	
	\subsubsection{Existence of the correlation functions}
	Our first result is a necessary and sufficient condition, the \textit{Seiberg bounds}, ensuring that the correlation functions defined by Equation~\eqref{eq:def_correl_intro}, are well-defined and non-trivial (\textit{i.e.} are finite and non-zero):
	\begin{thm}\label{thm:existence}
		For $F$ bounded and continuous over $H^{-1}(\Sigma\to\a,g)$, the correlation function $\ps{F(\Phi)\prod_{k=1}^NV_{\alpha_k}(z_k)\prod_{l=1}^MV_{\beta_l}(s_l)}_{g,\tau}$ exists and is non-trivial if and only if for $i=1,\cdots, r$:
		\begin{eqs} \label{eq:Seiberg}
			&\ps{\sum_{k=1}^N \alpha_k + \sum_{l=1}^M \frac{\beta_l}{2} - Q\chi(\Sigma),\omega_i^\vee} > 0;\\
			&\ps{\alpha_k - Q,e_i} < 0 \quad\text{ for all } k=1,...,N\quad\text{and}\quad\ps{\beta_l - Q,e_i} < 0 \quad\text{for all }l=1,...,M.
		\end{eqs}
	\end{thm}
	Here the $\omega_i^\vee$ form the basis of $\a$ dual to that of the simple roots $e_i$. These Seiberg bounds are the same as in the case of the sphere~\cite[Theorem 1.1]{CRV21}.
	
	\subsubsection{Conformal and diffeomorphism invariance}
	In addition to the existence of the correlation functions we would like to show that they satisfy basic assumptions related to conformal symmetry. However when $\tau\neq\Id$ this might no longer be the case since some of the $f_j$ no longer have same norm as the $e_i$: as such global conformal invariance is broken. To overcome this issue we need to assume that the GMC terms $e^{\ps{\frac{\gamma}2 f_j,\X^{g,C}}}\d \lambda_g$ do not appear in the path integral, which means taking $\mu_j=0$ for $j$ such that $\norm{f_j}^2\neq 2$. We refer to Subsection~\ref{subsec:different_models} for a more detailed discussion.
	Under this additional assumption we show a Weyl anomaly under a conformal change of metrics, and diffeomorphism invariance:
	\begin{thm} \label{thm:axioms}
		In the setting of Theorem~\ref{thm:existence}, assume that, for $\tau\neq\Id$, $\mu_j=0$ for $j$ such that $\norm{f_j}^2\neq 2$. Then for any $g'=e^\varphi g$ in the conformal class of $g$
		\begin{align} \label{eq:weyl}
			\ps{F(\Phi)\prod_{k=1}^NV_{\alpha_k}(z_k)\prod_{l=1}^MV_{\beta_l}(s_l)}_{g',\tau} = \ps{F(\Phi-\frac Q2\varphi)\prod_{k=1}^NV_{\alpha_k}(z_k)\prod_{l=1}^MV_{\beta_l}(s_l)}_{g,\tau}\times\\
			\prod_{k=1}^N e^{-\Delta_{\alpha_k}\varphi(x_k)}\prod_{l=1}^M e^{-\frac12\Delta_{\beta_l}\varphi(s_l)} \exp\left(\frac{r+6|Q|^2}{96\pi} \left(\int_{\Sigma} (|d_g\varphi|_g^2 + 2R_g\varphi) \d v_g + 4\int_{\partial\Sigma} k_g\varphi \d\lambda_g\right)\right). \nonumber
		\end{align}
		In the above the conformal weights are given by $\Delta_{\alpha}=\ps{\frac{\alpha}2,Q-\frac\alpha2}$.
		Besides, if $\psi : \Sigma \to \Sigma$ is an orientation-preserving diffeomorphism sending the boundary to itself:
		\begin{equation} \label{eq:diffeo invariance}
			\ps{F(\Phi)\prod_{k=1}^NV_{\alpha_k}(z_k)\prod_{l=1}^MV_{\beta_l}(s_l)}_{\psi^*g,\tau} = \ps{F(\Phi\circ\psi)\prod_{k=1}^NV_{\alpha_k}(\psi(z_k))\prod_{l=1}^MV_{\beta_l}(\psi(s_l))}_{g,\tau}.
		\end{equation}
	\end{thm}
	
	\begin{rem}\label{rem:central charge}
		The Weyl anomaly~\eqref{eq:weyl} allows one to identify the so-called central charge of the theory which is equal to $C = r + 6|Q|^2$, in agreement with the sphere case~\cite{CRV21}. 
	\end{rem}

    \subsubsection{Implications on the algebra of symmetry}
    The construction of the field of the theory has in turn some consequences on the symmetry algebra of the theory: these are $W$-algebras, vertex operator algebras containing the Virasoro algebra as a subalgebra. Indeed, by construction and in view of Equation~\eqref{eq:Phi_intro}, the underlying field of boundary Toda CFT satisfies $\Phi(\bar z)=\tau \Phi(z)$: this form of invariance actually lifts to the algebra of symmetry of our model. Namely we prove in Subsection~\ref{subsec:W_twist} (see Proposition~\ref{prop:W-alg} for a more precise statement) that:
    \begin{prop}\label{prop:W-alg_intro}
		Assume that $\g=\sl_n$. Then there exist currents $\Wb^{(s_i)}[\Psi]$ of spins $s_i=2,\cdots,n$, such that for generic values of $\gamma$:
		\begin{itemize}
			\item the $W$-algebra associated to $\g$ is generated using the modes of the $(\Wb^{(i)}[\Psi]])_{2\leq i\leq n}$;
			\item for $\tau\in\text{Out}(\mc R)$ non-trivial we have $\Wb^{(i)}[\tau\Psi]=\left(-1\right)^{i}\Wb^{(i)}[\Psi]$.
		\end{itemize}
	\end{prop}
    We actually have a more general statement for any $\g$.
    The fact that boundary Toda CFTs satisfy such a property, which in view of Equation~\eqref{eq:Phi_intro} synthetically takes the form  $\overline{\Wb^{(i)}}=\left(-1\right)^{i}\Wb^{(i)}$, was postulated in the $\sl_3$ case and in the conformal bootstrap setting in~\cite{FaRi}. In this respect our definition generalizes this feature to general $\g$ and $\tau$ and provides a concrete, path integral and probabilistic, realization of this property.
    
	\subsubsection{Further perspectives}
	The present document provides a basis for future studies of boundary Toda CFTs and paves the way for a mathematical understanding of these theories. In particular, thanks to the framework introduced in this document, we expect to address many questions related to the symmetries of the model, its solvability and its connection with classical problems. 
	\begin{enumerate}
		\item \textbf{Symmetries}: In two recent papers~\cite{CH_sym1, CH_sym2} we prove that for $\g=\sl_3$, $\Sigma=\H$ and $\tau=\Id$ the correlation functions introduced in this document are solutions of Ward identities associated to the $W_3$-algebra, which was previously unclear in the physics literature. Likewise we show that the model gives rise to \textit{higher equations of motion}, a feature which is new even from a physics perspective. This comes from the fact that for boundary Toda CFT singular vectors are not necessarily null, which translates as non-irreducibility of the corresponding $W_3$-modules.
		We hope to address more general cases (other automorphisms, surfaces, Lie algebras) in future works. 
		
		\item \textbf{Solvability}: Based on the higher equations of motion and the BPZ-type differential equations unveiled in~\cite{CH_sym2}, we aim to compute exact expressions for a family of structure constants for the $\g=\sl_3$ boundary Toda CFT, for which no expression has been proposed at the time being. This would involve in particular probabilistic techniques e.g. to make sense of Operator Product Expansions, combined with more algebraic tools coming from representation theory of the $W_3$-algebra. Thanks to the probabilistic setting introduced in this document we have already obtained new formulas for the reflection coefficients of boundary Toda CFTs. 
		
		\item \textbf{Classical theory}: 
		In this document we have proposed different constructions for boundary Toda CFTs using a path integral based on the Toda action. The critical points of this Toda action are solutions of the \textit{Toda equations}, which describe geometric problems related to uniformization of Riemann surfaces (for $\g=\mathfrak{sl}_2$) and embeddings into symmetric spaces (for general $\g$), (cyclic) Higgs bundles and integrable systems. When taking the parameter $\gamma$ to zero it is expected (under suitable assumptions) that the Toda field will concentrate around these critical points: this is the \textit{semi-classical limit}.
		
		The framework considered here suggests that in the boundary case the choice of an automorphism $\tau\in\text{Out}(\mc R)$ will give rise to a solution of the Toda equation: to the best of our knowledge this feature is new compared to the existing literature on Toda systems. We hope to prove this fact via the semi-classical limit $\gamma\to0$ of the probabilistic correlation functions introduced in this document.
	\end{enumerate}
	
	\textbf{Acknowledgments} B.C. would like to thank the Aix-Marseille University where part of this work has been undertaken and acknowledges support from the Eccellenza grant 194648 of the Swiss National Science Foundation; he is a member of NCCR SwissMAP. N.H. is grateful to Rémi Rhodes for having introduced him to conformal field theory and for many enlightening discussions. Part of this document was (re)written during the trimester program \lq\lq Probabilistic methods in quantum field theory" organized at the Hausdorff Research Institute for Mathematics.
	
	\section{Background and notations}
	
	The purpose of this section is to provide the reader with the necessary tools to make sense of Toda conformal field theories. We recall some notions related to conformal geometry in dimension two as well as complex simple Lie algebras.
	
	\subsection{Compact Riemann surfaces, Laplacians and Green's functions}\label{subsec:uniform}
	
	\subsubsection{Uniform metrics}
	
	Throughout the paper $\Sigma$ will be a compact connected Riemann surface with (possibly empty) boundary $\partial\Sigma$, and we denote by $\g(\Sigma)$ its genus and by $\k(\Sigma)$ the number of connected components of $\partial\Sigma$. The Euler characteristic is then defined by $\chi(\Sigma) = 2(1-\g(\Sigma)) - \k(\Sigma) \le 2$. The surfaces of Euler characteristic $<0$ are called of hyperbolic type, and the only (up to biholomorphism) surfaces of Euler characteristic $\ge 0$ are the torus, the sphere, the disk and the cylinder. In the construction of Toda CFT we will consider hyperbolic surfaces, but the case of the disk (or equivalently the complex upper half-plane) is very similar and only requires minor modifications. 
	
	The Gauss-Bonnet theorem for a compact Riemann surface equipped with a smooth Riemannian metric $g$ states that
	\begin{equation} \label{eq:gauss-bonnet}
		\int_\Sigma R_g \d v_g + 2\int_{\partial\Sigma} k_g \d \lambda_g = 4\pi \chi(\Sigma)
	\end{equation}
	where $R_g$ is the Ricci scalar curvature (that equals twice the Gaussian curvature) and $k_g$ the geodesic curvature along the boundary. Two smooth Riemannian metrics $g$ and $g'$ are said to be conformally equivalent if there exists a smooth function $\varphi$ on $\Sigma$ such that $g' = e^\varphi g$. The set of all metrics conformally equivalent to a given metric $g$ is called the conformal class of $g$ and is denoted $[g]$. We recall here the following useful identities:
	\begin{equation} \label{eq:conf change curvature}
		R_{g'} = e^{-\varphi}(R_g - \Delta_g\varphi) \quad ; \quad k_{g'} = e^{-\varphi/2}(k_g + \frac12\partial_{\vec{n}_g} \varphi)
	\end{equation}
	where $\Delta_g$ is the Laplacian and $\partial_{\vec{n}_g}$ is the outward normal derivative (the Neumann operator), defined by $\partial_{\vec{n}_g} f = (d_gf,\vec{n}_g)_g$ where $\vec{n}_g$ is the outward normal vector along $\partial\Sigma$, $d_g$ the gradient and $(\cdot,\cdot)_g$ the scalar product in the metric $g$. For a surface without boundary, the uniformization theorem implies that in each conformal class there exists a unique (up to scaling and isometry) metric of constant Gaussian curvature called \textit{uniform metric}. For a surface with boundary, it is shown in~\cite{OPS88} that in each conformal class there exist:
	\begin{itemize}
		\item a unique uniform metric of type 1, that is with constant Gaussian curvature and $0$ geodesic curvature along the boundary;
		\item a unique uniform metric of type 2, which means that it is flat (i.e. $R_g=0$) and has constant geodesic curvature.
	\end{itemize}
	Again uniqueness has to be understood up to scaling and isometry.
	
	We now describe the doubling of a hyperbolic surface with boundary. Let $\Sigma$ be a hyperbolic Riemann surface. We equip $\Sigma$ with a uniform type 1 metric $g$. It follows from the Gauss-Bonnet theorem that the metric $g$ is itself hyperbolic, \textit{i.e.} $R_g$ is a negative constant (that one usually chooses equal to $-2$ but this is not relevant for the present discussion). It is readily seen that there exists a closed surface $\hat{\Sigma} = \Sigma \sqcup \partial \Sigma \sqcup \sigma(\Sigma)$ such that $\sigma$ is an involution that preserves the boundary: $\sigma \circ \sigma = \Id$ and $\sigma_{|\partial\Sigma} = \Id$. (Intuitively, one \lq\lq glues" $\Sigma$ with a copy of itself along the boundary). We naturally define the metric on $\sigma(\Sigma)$ as the pullback of $g$ by $\sigma$. Further, since the boundaries of $\Sigma$ and $\sigma(\Sigma)$ are geodesics of same length, the metric $g$ extends to a metric $\hat{g}$ on $\hat{\Sigma}$ given on $\sigma(\Sigma)$ by the pullback $\sigma^*g$. Since $g$ is smooth on $\Sigma$, the extended metric $\hat{g}$ is smooth away from $\partial\Sigma$. Near each point $x$ in $\partial\Sigma$, since the metric $g$ is hyperbolic, there is an isometry from the geodesic ball $B_{\hat{g}}(x,\eps)$ (with $\eps>0$ small enough) to the hyperbolic ball $B_{g_P}(i,\eps)$ in the complex upper half-plane $\H$, such that the $\sigma(\Sigma)$-part of the geodesic ball in $\hat{\Sigma}$ is sent to the left-hand side of the imaginary axis in $\H$ and the $\Sigma$-part of the ball is sent to the right of the imaginary axis. This provides local complex coordinates $z$ so that $\hat{g} = g_P = |dz|^2/\Im(z)^2$ near $x$, showing that $\hat{g}$ is smooth on $\hat{\Sigma}$. Note that for the sake of simplicity we will often omit the hat in the sequel, and denote by $g$ the induced smooth metric on the doubled surface.
	
	\subsubsection{Green's functions} \label{sec: green_functions}
	
	In this paper we will be interested in three different PDE problems on a compact Riemann surface equipped with a smooth Riemannian metric $g$. If the boundary of the surface is empty then we will be interested in what we call the closed Poisson problem, while if the boundary is non-empty we have to specify boundary conditions, that will be either of Neumann or Dirichlet type. In this section we define the Green's functions associated with these problems and state some useful properties. Recall that $\Delta_g$ denotes the (negative) Laplacian in the metric $g$.
	
	Suppose first that $\partial\Sigma = \emptyset$. The Green's function of the closed Poisson problem on $(\Sigma,g)$ is defined as the solution (see e.g.~\cite{GRV19} for more details) $G_g$ of the weak formulated problem\footnote{The factor $2\pi$ in the equation is cosmetic and added in order to lighten the notations in the sequel.}
	\begin{equation}\label{eq:closed_pb_green}
		\left\lbrace \begin{array}{ll}
			-\Delta_g G_g(\cdot,y) = 2\pi(\delta_y - \frac{1}{v_g(\Sigma)}) & \text{in } \Sigma \\
			m_g(G_g(\cdot,y)) = 0 & \text{for all } y \in \Sigma
		\end{array} \right.
	\end{equation}
	where $\delta$ is the Dirac delta function and $m_g(f)$ is the average of $f$ over $\Sigma$, with respect to the metric $g$: $m_g(f) := v_g(\Sigma)^{-1}\int_\Sigma f\d v_g$. As it is defined $G_g$ is \emph{a priori} a Schwartz distribution on $\Sigma\times\Sigma$, but we can see it as the integral kernel of a smoothing operator from $L^2_0(\Sigma)$, the space of square-integrable functions with vanishing mean on $\Sigma$, to $H^2(\Sigma)$ the second-order $L^2$-Sobolev space on $\Sigma$, through the natural $L^2$-pairing. In addition, for any function $f$ smooth over $\bar\Sigma\coloneqq\Sigma\sqcup\partial\Sigma$,
	\begin{equation}\label{eq:gauss-green formula closed}
		-\int_{\Sigma}G_g(x,\cdot)  \Delta_g f \d v_g = 2\pi(f(x) - m_g(f)).
	\end{equation}
	Further, $G_g$ is symmetric. As we will see below (Lemma~\ref{lemma:estimates_green_function}), $G_g$ is actually a smooth function away from the diagonal, with logarithmic singularity on the diagonal. Finally, if $g'$ is in the conformal class of the metric $g$ then the function $G_{g'}$ defined for $x\neq y \in \Sigma$ by~\footnote{In fact this defines $G_{g'}$ in the sense of integral kernels on the full $\Sigma\times\Sigma$.}
	\begin{equation} \label{eq:conf change green functions}
		G_{g'}(x,y) := G_g(x,y) - m_{g'}(G_g(x,\cdot)) - m_{g'}(G_g(\cdot,y)) + m_{g'}(G_g) \qt{where}
	\end{equation}
	\begin{equation*}
		m_{g'}(G_g):=\frac{1}{v_{g'}(\Sigma)^2} \iint_{\Sigma\times\Sigma} G_g(x,y) \d v_{g'}(x) \d v_{g'}(y)
	\end{equation*}
	solves the problem \eqref{eq:closed_pb_green} in the metric $g'$. 
	
	We now turn to the case where $\partial\Sigma \neq \emptyset$. The Neumann and Dirichlet Green's functions are defined respectively as the solutions of
	\begin{equation}\label{eq:Neumann_pb_green}
		\left\lbrace \begin{array}{ll}
			-\Delta_g G^N_g(\cdot,y) = 2\pi(\delta_y - \frac{1}{v_g(\Sigma)}) & \text{in } \Sigma \\
			\partial_{\vec{n}_g} G^N_g(x,y) = 0 & \text{for } x\in \partial\Sigma \text{ and for all } y\in\Sigma \\
			m_g(G^N_g(\cdot,y)) = 0 & \text{for all } y\in\Sigma
		\end{array} \right.
	\end{equation} 
	where recall that $\partial_{\vec{n}_g}$ is defined with respect to the outward normal vector $\vec{n}_g$, and
	\begin{equation}\label{eq:dirich_pb_green}
		\left\lbrace \begin{array}{ll}
			-\Delta_g G^D_g(\cdot,y) = 2\pi\delta_y & \text{in } \Sigma \\
			G^D_g(x,y) = 0 & \text{for } x\in \partial\Sigma \text{ and for all } y\in\Sigma.
		\end{array} \right.
	\end{equation}
	In the same way as in the closed case we see the Neumann and Dirichlet Green's functions as integral kernels of operators from an $L^2$ space to a second order Sobolev space on $\Sigma$. As such, for any $f$ smooth over $\bar\Sigma=\Sigma\cup\partial\Sigma$, the Neumann Green's function satisfies:
	\begin{equation} \label{eq:gauss-green formula neumann}
		\int_{\partial\Sigma}  G^N_g(x,\cdot)\partial_{\vec{n}_{g}}f \d \lambda_g - \int_{\Sigma} G^N_g(x,\cdot)\Delta_g f  \d v_g = 2\pi(f(x) - m_g(f))
	\end{equation}
	while the Dirichlet Green's function satisfies~\eqref{eq:gauss-green formula closed}. Both also enjoy the same change of metric formula~\eqref{eq:conf change green functions} as the closed Green's function $G_g$.
	
	We now provide a relation between the closed Green's function and the boundary ones, using the doubling of surface described in the previous subsection.
	\begin{lemma}\label{lemma:decomposition_green_function}
		Let $\Sigma$ be a compact surface of hyperbolic type with boundary $\partial\Sigma \neq \emptyset$, equipped with a uniform type 1 metric $g$. Let $\hat{\Sigma}$ be the doubled surface equipped with the smooth hyperbolic metric $\hat{g}$ and $\sigma$ be the canonical involution. Then for $x,y \in \overline{\Sigma}$ we have 
		\begin{equation} \label{eq:decomposition green neumann}
			G^N_g(x,y) = G_{\hat{g}}(x,y) + G_{\hat{g}}(x,\sigma(y))
		\end{equation}
		and
		\begin{equation} \label{eq:decomposition green dirichlet}
			G^D_g(x,y) = G_{\hat{g}}(x,y) - G_{\hat{g}}(x,\sigma(y)).
		\end{equation}
	\end{lemma}
	
	\begin{proof}
		We first have to show that the right-hand side of \eqref{eq:decomposition green neumann} is a solution to the problem \eqref{eq:Neumann_pb_green}. We compute the Laplacian in the variable $x$:
		$$
		-\frac{1}{2\pi}\Delta_{g} (G_{\hat{g}}(x,y) + G_{\hat{g}}(x,\sigma(y))) = \delta_y(x) - \frac{1}{v_{\hat{g}}(\hat{\Sigma})} + \delta_{\sigma(y)}(x) - \frac{1}{v_{\hat{g}}(\hat{\Sigma})} = \delta_y(x) - \frac{1}{v_{g}(\Sigma)}.
		$$
		Then we investigate the outer normal derivative along $\partial \Sigma$. Since it holds that $G_{\hat{g}}(x,y) = G_{\hat{g}}(\sigma(x),\sigma(y))$, then for $x\in\partial\Sigma$ and $y\in\Sigma$
		\begin{align*}
			&\partial_{\vec{n}_g} (G_{\hat{g}}(x,y) + G_{\hat{g}}(x,\sigma(y))) = \partial_{\vec{n}_g} G_{\hat{g}}(x,y) - \partial_{\vec{n}_{g}}G_{\hat{g}}(x,y) = 0 \quad\text{and likewise}\\
			&\int_\Sigma (G_{\hat{g}}(x,y) + G_{\hat{g}}(x,\sigma(y))) \d v_{g}(y) = \int_{\hat{\Sigma}} G_{\hat{g}}(x,y) \d v_{\hat{g}}(y) = 0.
		\end{align*}
		Combining these properties with the symmetry of the Green's functions implies the result. The proof of~\eqref{eq:decomposition green dirichlet} is immediate once one recalls that $\sigma = \Id$ on the boundary of $\Sigma$.
	\end{proof}
	
	Eventually we recall here some estimates from~\cite{GRV19} on the Green's function $G_{\hat{g}}$, which show that it has a logarithmic divergence on the diagonal:
	\begin{lemma} \label{lemma:estimates_green_function}
		On a compact surface $\Sigma$ without boundary equipped with a hyperbolic metric $g_0$, the Green's function $G_{g_0}$ can be written when $x$ and $y$ are closed to each other as
		\begin{equation} \label{eq:log divergence diagonal}
			G_{g_0}(x,y) = \log \frac{1}{d_{g_0}(x,y)} + f_{g_0}(x,y)
		\end{equation}
		with $f_{g_0}$ a smooth function on $\Sigma^2$. Moreover, for any $x_0$ in $\Sigma$ there are coordinates $z \in B(0,1) \subset \C$ such that $g_0 = \frac{4|dz|^2}{(1-|z|^2)^2}$ and 
		\begin{equation}
			G_{g_0}(z,z') = \log \frac{1}{|z-z'|} + F(z,z')
		\end{equation}
		near $x_0$, with $F$ smooth. Finally, if $g$ is another metric in the conformal class of the hyperbolic metric $g_0$, then the estimate~\eqref{eq:log divergence diagonal} holds near the diagonal with $f_{g}$ continuous.
	\end{lemma}
	This logarithmic divergence allows to make sense of the Gaussian Multiplicative Chaos measures associated to the Gaussian Field with covariance kernel $G_g$ in subsequent sections.
	
	\subsubsection{Regularized determinants} \label{subsec:determinants}
	
	When dealing with Gaussian measures on finite-dimensional spaces, one has to compute determinants of the inverse of covariance matrices. In the infinite-dimensional setting, we will see that the natural candidate to generalize this picture is the Gaussian Free Field, whose covariance structure is given by the Green's function, so that one is naturally led to consider determinants of Laplacians. In order to define the determinant of the Laplacian on a compact and connected Riemann surface $\Sigma$ one needs a regularization procedure. To this aim the classical method is to introduce the spectral $\zeta$-function
	$$
	\zeta(s) = \sum_{\lambda>0}\lambda^{-s}
	$$ 
	where the $\lambda$'s are the eigenvalues of the (positive) Laplacian~\footnote{So for us, $-\Delta_g$ on closed surfaces, and $-\Delta_g$ with Neumann or Dirichlet conditions on surfaces with boundary, which we will denote later $-\Delta_g^N$ and $-\Delta_g^D$.}.  On a compact surface, the spectrum of the positive Laplacian is discrete and forms an increasing sequence $0 \leq \lambda_0 < \lambda_1 \leq ... \to \infty$, with $\lambda_0=0$ for the closed and the Neumann Laplacian, and $\lambda_0>0$ for the Dirichlet Laplacian. The asymptotics of the eigenvalues are given by the Weyl's law (see e.g. \cite{weyl_law_ivrii}), which allows one to prove that the spectral $\zeta$-function converges for $\Re(s)$ positive enough, and admits a meromorphic continuation to $\Re(s)>-1$ that is regular at $0$ (see~\cite{OPS88}). We then define the \emph{regularized determinant} as
	$$
	\text{det}^\prime(-\Delta_g) = e^{-\zeta'(0)}
	$$
	and simply write $\det$ for the Dirichlet Laplacian since the full spectrum is involved in the determinant. The Polyakov-Alvarez formula gives the behavior of the determinant under a conformal change of metric, and will be very useful to compute the conformal anomaly of Toda correlation functions. The formula was originally derived in the closed case by Polyakov~\cite{polyakov_bosonic_strings,polyakov_fermionic_strings} and refined in the boundary case by Alvarez~\cite{Alv83} (see also~\cite{McKeanSinger}). We state it as a lemma, and refer to e.g. \cite{OPS88} for a concise exposition of this kind of results.
	\begin{lemma}\label{lemma:background/polyakov formula}
		Let $g' = e^\varphi g$ be a smooth Riemannian metric in the conformal class of $g$. Then we have the following relation between determinants of the Laplacians associated with $g$ and $g'$:
		\begin{equation}
			\log \frac{\det\left(-\Delta_{g'}\right)}{\det\left(-\Delta_{g}\right)}=-\frac{1}{48\pi} \left(\int_{\Sigma} (|d_g\varphi|_g^2 + 2R_g\varphi) \d v_g + 4\int_{\partial\Sigma} k_g\varphi \d\lambda_g\right) \pm \frac{1}{8\pi} \int_{\partial\Sigma} \partial_{\vec{n}_{g}} \varphi \d\lambda_g
		\end{equation}
		where $\det(-\Delta_g):=\det'(-\Delta_g)/v_g(\Sigma)$ for the closed Laplacian and the Neumann Laplacian. The sign of the corrective term is positive in case of Neumann boundary conditions and negative in case of Dirichlet boundary conditions.\footnote{Our convention for the conformal factor differs from the standard convention by a factor $2$, i.e. $\varphi=2\omega$ where $\omega$ stands for the usual conformal factor, which explains that the formula is slightly different.}
	\end{lemma}
	
	We will also consider the regularized determinant $\det^\prime\left(-\frac1{2\pi}\Delta_g\right)$ to define the partition function of the Gaussian Free Field. According to the previous discussion we have  
	\begin{equation}
		{\det}^\prime\left(-\frac1{2\pi}\Delta_g\right)=\left(2\pi\right)^{-\zeta_0}{\det}^\prime\left(-\Delta_g\right)
	\end{equation}
	where $\zeta_0$ is independent of the metric $g$ and given by $\zeta_0=\frac{\chi(\Sigma)}{6}-1$ for a closed surface and in the Neumann case while in the Dirichlet case the determinants are not truncated and $\zeta_0=\frac{\chi(\Sigma)}{6}$. Note that in the closed case $\zeta_0=\zeta(0)$. The above identity can be recovered using Lemma~\ref{lemma:background/polyakov formula} by taking $\varphi=\ln 2\pi$ together with the Gauss-Bonnet formula~\eqref{eq:gauss-bonnet}.
	
	\subsection{Finite-dimensional complex simple Lie algebras}\label{subsec:Lie}
	
	The definition of Toda CFTs relies on a Lie algebra structure. We gather in this section some classical facts about finite-dimensional complex simple Lie algebras, and introduce the notations that will be used throughout the paper. Additional details can be found \textit{e.g.} in the textbook~\cite{Hum72}.
	
	\subsubsection{Basic facts}
	
	In virtue of the Cartan-Killing classification, every finite-dimensional complex simple Lie algebra is isomorphic to either one of the following family of classical Lie algebras: $A_n$ for $n\geq 1$ (corresponding to $\mathfrak{sl}_{n+1}(\C)$), $B_n$ for $n\geq 2$ (corresponding to $\mathfrak{so}_{2n+1}(\C))$, $C_n$ for $n\geq 3$ (corresponding to $\mathfrak{sp}_{2n}(\C)$) and $D_n$ for $n\geq 4$ (corresponding to $\mathfrak{so}_{2n}(\C)$), or one of the five exceptional simple Lie algebras: $E_6$, $E_7$, $E_8$, $F_4$ and $G_2$. For any finite-dimensional complex simple Lie algebra $\mathfrak{g}$ there is a natural Euclidean space $\mathfrak{a}$ that is such that the Cartan subalgebra of $\mathfrak{g}$ decomposes as $\mathfrak{a}+i\mathfrak{a}$ (this space $\a$ is the real Cartan subalgebra of the split real form of $\g$). To this Euclidean space is attached a scalar product $\ps{\cdot,\cdot}$ which is proportional to the Killing form of $\mathfrak{g}$ restricted to $\a$. The Euclidean space $(\mathfrak{a},\ps{\cdot,\cdot})$ is unique up to isomorphism and can be realized as $\R^r$ equipped with its standard scalar product, where $r$ is the rank of $\mathfrak{g}$. 
	
	The so-called simple roots $(e_i)_{i=1,...,r}$ form a particular basis of $\mathfrak{a}$; this basis is such that $\ps{e_i^\vee,e_j} = A_{i,j}$ where $A$ is the Cartan matrix of $\mathfrak{g}$ and $e_i^\vee := 2\frac{e_i}{\ps{e_i,e_i}}$ is the coroot of $e_i$. In the sequel we will assume that the scalar product is renormalized in such a way that the longest simple root has squared length $2$, which is a standard assumption. The dual basis of $(e_i^\vee)_{1\leq i\leq r}$ is given by $(\omega_i)_{1\leq i\leq r}$ where the
	$$
	\omega_i \coloneqq \sum_{j=1}^r (A^{-1})_{j,i} e_j\qt{for}i=1,\cdots,r
	$$
	are called the fundamental weights. They verify $\ps{e_i^\vee,\omega_j} = \delta_{ij}.$ 
	Another important object is the Weyl vector, defined by $\rho = \sum_{i=1}^r \omega_i$, which is such that $\ps{\rho,e_i^\vee}=1$ for all $i$. We also define the Weyl vector associated to the coroots by $\rho^\vee = \sum_{i=1}^r \omega_i^\vee$ where $\omega_i^\vee$ is defined so that $\ps{\omega_i^\vee,e_j} = \delta_{ij}$ for $i=1,...r$.
	
	\subsubsection{The outer automorphism group}\label{subsec:out}
	
	The notion of outer automorphism~\cite[Section 12.2]{Hum72} naturally arises when discussing the admissible boundary conditions of the theory, as discussed in the introduction.
	
	Let $\g$ be a finite-dimensional complex simple Lie algebra with associated root system $\mc R$ (such a root system characterizes the Lie algebra $\g$ up to isomorphism). The automorphism group of $\mc R$ is a semi-direct product $\text{Aut}(\mc R) = \text{Out}(\mc R) \ltimes W(\g)$, where elements of $\text{Out}(\mc R)$ are called outer automorphisms and correspond to Dynkin diagram automorphisms, while $W(\g)$ is the Weyl group of $\g$, i.e. automorphisms of $\mc R$ induced by inner automorphisms of $\g$ (generated by elements of the form Ad$(g)$ where $g$ lies in a Lie group with Lie algebra $\g$ and Ad is the adjoint map). Standard results~\cite[Section 12.2]{Hum72} then show that for a root system associated with a finite-dimensional complex simple Lie algebra $\mathfrak{g}$ the outer automorphism group is either: trivial, this is the case for $A_1$, $B_n$, $C_n$, $E_7$, $E_8$, $F_4$ and $G_2$, or isomorphic to the cyclic group of order two in all the other cases, except for $D_4$ where the outer automorphism group is isomorphic to the group of permutations on three elements (thus of order 6). In all cases an element $\tau$ in $\text{O}(\mc R)$ is an isometry of $\a$, i.e. $\ps{u,v}=\ps{\tau(u),\tau(v)}$ for any $u,v\in\a$.
	
	As an example, for the Lie algebra $A_n$ ($\sl_{n+1}$) the non-trivial automorphism $\tau$ is given by $e_i \mapsto e_{n+1-i}$, so that for $A_2$ it swaps the two simple roots: if $u = u_1e_1 + u_2e_2 \in \mathfrak{a}$, then $\tau(u) = u_2e_1 + u_1e_2$. In general the outer automorphisms are particular permutations of the simple roots. 
	
	\subsubsection{Folding of root systems}\label{subsec:fold}
	Given $\tau$ in $\text{Out}(\mc R)$ let us denote by $\a_N\subseteq\a$ the corresponding invariant subspace and by $d_N$ its dimension. There is a natural root system in $\a_N$ defined by considering, with the $(\mc O_j)_{1\leq j\leq d_N}$ being the orbits of $\tau$ (viewed as acting on $\{1,\cdots,r\}$):
	\begin{equation}
		f_j\coloneqq \frac{1}{\norm{\mc O_j}}\sum_{i\in \mc O_j}e_i\qt{for $1\leq j\leq d_N$}
	\end{equation}
	Indeed one can check that $\ps{f_i^\vee,f_j}=A_{i,j}^\tau$ where $A^\tau$ is the Cartan matrix of the complex simple Lie algebra $\g^\tau$ which is the \textit{folding} of $\g$. As a consequence $(f_j)_{1\leq j\leq d_N}$ can be viewed as simple roots for the \textit{folded} root system $\Phi^\tau$ of $\g^\tau$, with the length of the longest root being now given by some $\kappa$. The table below describes this folding in all non-trivial cases.
	\begin{center}
		\begin{tabular}{ |c|c|c|c|c| } 
			\hline
			Root system $\mc R$ & Order of $\tau\in\text{Out}(\mc R)$ & Invariant sub-root system & Dimension of $\a_N$ & $\kappa^2$  \\ 
			\hline
			$A_{2n}$ & 2 & $B_n$ & $n$ & $1$\\ 
			$A_{2n-1}$ & 2 & $C_n$ & $n$ & $2$ \\ 
			$D_n, n\geq 4$ & 2 & $B_{n-1}$ & $n-1$ & $2$ \\ 
			$D_4$ & 3 & $G_2$ & $2$ & $2$\\
			$E_6$ & 2 & $F_4$ & $4$ & $2$\\
			\hline
		\end{tabular}
	\end{center}
	For outer automorphisms of order two the $f_j$'s are of the form $f_j=e_i+\tau(e_i)$ or $f_j=\frac{e_i+\tau(e_i)}2$, so that $\mathfrak{a}_N = \text{Span}(v+\tau(v), v \in \mathfrak{a})$. Moreover such automorphisms are self-adjoint with respect to the scalar product on $\mathfrak{a}$, which allows to decompose the space $\mathfrak{a}$ as the orthogonal sum $\mathfrak{a}=\mathfrak{a}_N \bigoplus \mathfrak{a}_D$ where $\mathfrak{a}_D = \text{Span}(v-\tau(v), v\in\mathfrak{a})$. For the two 3-cycles of $\text{Out}(D_4)$, $\mathfrak{a}=\mathfrak{a}_N \bigoplus \mathfrak{a}_D$ with $\mathfrak{a}_N = \text{Span}\left(v+\tau(v)+\tau^2(v), v \in \mathfrak{a}\right)$ and $\mathfrak{a}_D = \text{Span}(v-\tau(v),v-\tau^2(v), v\in\mathfrak{a})$. 
	
	It is important to note that in all non-trivial cases the Weyl vector $\rho$ belongs to the invariant subspace $\a_N$: this is due to the fact that for simply-laced Lie algebras it is characterized by $\ps{\rho,e_i}=1$ for all $1\leq i\leq r$, a condition which is preserved by the action of an outer automorphism. Hence the background charge $Q=\gamma\rho+\frac2\gamma\rho^\vee$ always belongs to $\a_N$.
	
	Hereafter for notational simplicity the outer automorphism under consideration will often be denoted by a star $\cdot^*$. Likewise we will denote by $p_N$ (resp. $p_D$) the orthogonal projection on $\a_N$ (resp. $\a_D$). We also set for $\tau$ non-trivial 
	\begin{equation} \label{eq:Q tau gamma tau}
		Q_\tau\coloneqq \gamma\rho_\tau+\frac2{\gamma}\rho,\text{where  $\ps{\rho_\tau,f_j}=\frac{\norm{f_j}^2}{2}$ for all $1\leq j\leq d_N$},
	\end{equation}
	and introduce the notation $I_\tau\coloneqq\left\{j\in\{1,\cdots,d_N\},\quad\norm{f_j}^2=2\right\}$.

	\subsubsection{Semi-simple Lie algebras} Any finite-dimensional complex \textit{semi-simple} Lie algebra $\g$ decomposes as a direct sum $\g=\bigoplus_{k=1}^p \g_k$ where the $\g_k$'s are finite-dimensional complex \textit{simple} Lie algebras. As pointed out in~\cite[Remark 1.2]{CRV21}, if $\g=\g_1\oplus\g_2$ is a direct sum of semi-simple Lie algebras then the Toda correlation functions factorize as $\ps{\V}^\g=\ps{\V}^{\g_1}\ps{\V}^{\g_2}$ (and the Weyl anomaly thus shows that the central charges add up, see Remark~\ref{rem:central charge}). This allows one to reduce the study of Toda theories for semi-simple Lie algebras to those with respect to simple Lie algebras. We thus work with $\g$ simple in the sequel. 
	
	\section{Toda theories on closed Riemann surfaces} \label{sec:closed}
	
	In this section, we propose a mathematical construction of Toda CFTs based on a compact hyperbolic Riemann surface without boundary, and associated to any complex simple Lie algebra $\g$. For this purpose we rely on a probabilistic framework based on Gaussian Free Fields (GFFs hereafter) and Gaussian Multiplicative Chaos (GMC in the sequel).
	Throughout this section, $\Sigma$ will denote a connected compact Riemann surface without boundary, assumed to be of hyperbolic type in the sense that $\chi(\Sigma) < 0$. We equip $\Sigma$ with a smooth Riemannian metric $g$. 
	
	\subsection{Gaussian Free Field} \label{sec:GFF closed}
	The action functional entering the path integral definition of Toda CFTs~\eqref{eq:toda path integral} is made of several terms. Our first goal is to interpret the quadratic component, that is, to make sense of 
	\begin{equation} \label{eq:path integral GFF}
		\int F(\bphi) e^{-\frac{1}{4\pi}\int_{\Sigma} |d_g\bphi|_g^2 \d v_g} D_g\bphi
	\end{equation} 
	where in the above the gradient $d_g$ acts component by component on functions from $\Sigma$ to $\a$. This is achieved by the introduction of the vectorial Gaussian Free Field (GFF hereafter). Indeed, the formal integral \eqref{eq:path integral GFF} can be written as
	\begin{equation} \label{eq:path integral GFF gaussian}
		\int F(\bphi) e^{-\frac12 \int_{\Sigma} \ps{\bphi,-\frac{1}{2\pi}\Delta_g\bphi}_g \d v_g} D_g\bphi	
	\end{equation}
	where, again, the Laplacian acts independently on the $r$ components of $\bphi$. Equation~\eqref{eq:path integral GFF gaussian} describes (formally) the law of a Gaussian field with values in $\a$, mean 0 and positive definite covariance kernel $(-\frac{1}{2\pi}\Delta_g)^{-1}$. There is one subtlety here in that this observation is valid only on the orthogonal of the kernel of the Laplacian, so that the above Gaussian field would be defined only on this subspace. We will come back to this issue later. The inverse of the one-dimensional Laplacian (outside of its kernel) is given in the sense of integral kernels by the Green's function, and since the $r$-dimensional Laplacian in \eqref{eq:path integral GFF gaussian} acts independently on the $r$ components of the field, it is natural to define our underlying field as follows. 
	\begin{defdef}
		We set $\X^g$ to be the centered Gaussian field on $\Sigma$ with covariance structure
		$$
		\E[(\ps{u,\X^g},f)_{H^1(\Sigma,g)} (\ps{v,\X^g},f')_{H^1(\Sigma,g)}] = \ps{u,v} \int_{\Sigma^2} G_g(x,y) f(x) f'(y) \d v_g(x) \d v_g(y)
		$$
		for all $u,v\in\a$ and $f,f' \in H^1(\Sigma,g)$, where $(\ps{u,\X^g},f)_{H^1(\Sigma,g)}$ is the pairing between $H^1$ and its dual $H^{-1}$.
	\end{defdef}
	Indeed, the field $\X^g$ is a random distribution, and lives almost surely in the negative index Sobolev space $H^{-1}(\Sigma\to\a,g)$, which means that each component is in $H^{-1}$. As already mentioned, this Gaussian field is only defined up to constants, and we need to take this fact into account in order to define the path integral~\eqref{eq:path integral GFF}. We overcome this issue by tensorizing the law of the GFF by the Lebesgue measure on $\R^r$. That is, we interpret \eqref{eq:path integral GFF gaussian} as
	\begin{equation} \label{eq:interpretation path int gff}
		\int F(\bphi) e^{-\frac{1}{4\pi}\int_{\Sigma} |d_g\bphi|_g^2 \d v_g} D_g\bphi \propto \int_\a \E[F(\X^g+\bm{c})] \d\bm{c}
	\end{equation}
	with the expectation taken with respect to the GFF and the integration over $\a$ defined by 
	$$
	\int_\a F(\bm{c}) \d\bm{c} = \int_{\R^r} F\Big(\sum_{i=1}^r \bm{v}_i c_i\Big) \prod_{i=1}^r \d c_i
	$$
	where $(\bm{v}_i)$ is any orthonormal basis of $\a$. The theory of GFF is well-known and we refer for example to \cite{Ber17,PW21} for a review. For the sake of simplicity we will use the notation
	\begin{equation} 
		\E[\ps{u,\X^g(x)}\ps{v,\X^g(y)}] = \ps{u,v} G_g(x,y)
	\end{equation}
	hereafter. Note that $\X^g$ can be represented as the sum $\X^g=\sum_{i=1}^rX_i^g\bm v_i$ where the $(X_i^g)_{1\leq i\leq r}$ are independent, real-valued GFFs over $\Sigma$, and $(\bm v_i)_{1\leq i\leq r}$ is any orthonormal basis of $\a$. 
	
	For the picture to be complete, it misses a multiplicative normalization term called the partition function: this corresponds to the total mass of the path integral \eqref{eq:path integral GFF}. Following the analogy with the finite-dimensional case mentioned in Subsection~\ref{subsec:determinants}, we define it using the regularized determinant of the Laplacian. 
	\begin{defdef}\label{def:partition function GFF 2D}
		The partition function of the GFF $\X^g$ is defined by
		$$
		\mathcal{Z}_\X(\Sigma,g) = \left(\frac{\det{'}\left(-\frac{1}{2\pi}\Delta_g\right)}{ v_g(\Sigma)}\right)^{-\frac{r}{2}}.
		$$
		where $\Delta_g$ is the (negative) Laplacian acting on functions from $\Sigma$ to $\R$.
	\end{defdef}
	An important property of the GFF is its behavior under conformal changes of metric. This has a direct implication for the measure defined by Equation~\eqref{eq:interpretation path int gff}.
	\begin{lemma} \label{lemma:change of metric GFF}
		Let $g'$ be conformally equivalent to the metric $g$. Then for all positive continuous functional $F$ over $H^{-1}(\Sigma\to\a,g)$, 
		\begin{equation} \label{eq:conf change GFF}
			\X^{g'} \eqlaw \X^g - m_{g'}(\X^g),\qt{so that}\int_{\a} \E[F(\X^{g'} + \c)] \d\c = \int_{\a} \E[F(\X^{g} + \c)] \d\c.
		\end{equation}
		
	\end{lemma}
	\begin{proof}
		Let $\tilde{\X}^g \coloneqq \X^g - m_{g'}(\X^g)$. For any $u,v \in \a$ and $x,y \in \Sigma$ we have
		\begin{align*}
			\E[\ps{u,\tilde{\X}^g(x)}\ps{v,\tilde{\X}^g(y)}] = \ps{u,v}&\left(G_g(x,y) - m_{g'}(G_g(x,\cdot)) - m_{g'}(G_g(\cdot,y)) \right.\\&\left.+ v_{g'}(\Sigma)^{-2}\iint_{\Sigma^2} G_g(x,y) \d v_{g'}(x) \d v_{g'}(y)\right)
		\end{align*}
		which is equal to $\ps{u,v}G_{g'}(x,y)$ thanks to Equation~\eqref{eq:conf change green functions}. Since both fields are centered Gaussian fields this is enough to conclude that $\tilde{\X}^g$ has the same law as $\X^{g'}$. Now to prove the lemma we make a change of variables in the zero-mode:
		$$
		\int_{\a} \E[F(\X^{g'} + \c)]\d\c = \int_{\a} \E[F(\X^{g} - m_{g'}(\X^g) + \c)]\d\c = \int_{\a} \E[F(\X^{g} + \c)]\d\c .
		$$
	\end{proof}
	Further, as a straightforward consequence of Lemma~\ref{lemma:background/polyakov formula} one obtains the following conformal anomaly for the partition function of the GFF.
	
	\begin{prop} \label{eq:conformal anomaly X}
		Let $g' = e^\varphi g$ be a metric in the conformal class of $g$. Then we have
		\begin{equation} 
			\mathcal{Z}_\X(\Sigma,g') = \mathcal{Z}_\X(\Sigma,g) \exp\left(\frac{r}{96\pi} \int_{\Sigma} (|d_g\varphi|_g^2 + 2R_g\varphi) \d v_g\right).
		\end{equation}
	\end{prop}
	
	\subsection{Gaussian Multiplicative Chaos} \label{subsec:GMC_closed}
	
	Having given a meaning to the quadratic term that appears in the action~\eqref{eq:toda_action}, we now would like to give a rigorous interpretation of the exponential potential that appears there, \emph{i.e.} define the exponentials $e^{\ps{u,\X^g}} \d v_g$ for $u\in\a$. However, since $\X^g$ is far from being a continuous well-defined function we need to introduce a regularization procedure. This method is at the heart of the theory of Gaussian Multiplicative Chaos --- for details about the theory of GMC, see for instance \cite{RV10, RV14, Ber17}. 
	
	\begin{defdef}
		Given $g$ a smooth Riemannian metric on $\Sigma$ and $x\in\Sigma$, let
		$C_g(x,\eps)\coloneqq\left\{y\in\Sigma, d_g(x,y)=\eps\right\}$ be the geodesic circle centered at $x$ of (small) radius $\eps>0$. We define the regularized field $\X^g_\eps$ at $x$ as the average of the GFF $\X^g$ over $C_g(x,\eps)$.
	\end{defdef}
	Synthetically, $\X^g_\eps(x)=\mu_{g,\eps}[\X_g](x)$ where $\mu_{g,\eps}[\cdot](x)$ is the uniform probability measure on $C_g(x,\eps)$, that is $\int \cdot \d\mu_{g,\eps}(x) = \frac{1}{\lambda_g(C_g(x,\eps))} \int_{C_g(x,\eps)} \cdot \d\lambda_g$. It can be shown~\cite{DS11} that the random variable $\X^g_\eps(x)$ has an almost surely continuous version that we will consider from now on. 
	
	\begin{lemma}\label{lemma:regularization_X}
		The field $\X^g_\eps$ has the following covariance structure:
		\begin{equation}\label{eq:cov_Xeps}
			\E[\ps{u,\X^g_\eps(x)}\ps{v,\X^g_\eps(y)}] = \ps{u,v} \iint G_g \d\mu_{g,\eps}(x) \d\mu_{g,\eps}(y)
		\end{equation}
		for any $u, v \in \a$. Moreover, if $g =e^{\varphi}g_0$ with $g_0$ hyperbolic then as $\eps\to 0$ and uniformly in $x$
		\begin{equation}\label{eq:variance_Xeps}
			\E[\ps{u,\mu_{g,\eps}[\X^{g_0}](x)}^2] = |u|^2\left(-\log \eps + W_{g_0}(x)+\frac12\varphi(x)\right) + o(1)
		\end{equation}
		with $W_{g_0}$ a smooth function on $\Sigma$ independent from $\varphi$.
	\end{lemma}
	\begin{proof}
		The first item follows from the definition of $\X_{\eps}^g$. In the case where $\varphi=0$ the second item is a straightforward higher-rank generalization of Lemma 3.2 in~\cite{GRV19}. The case of non-zero $\varphi$ corresponds to~\cite[Equation (3.10)]{GRV19}.
	\end{proof}
	From Lemma~\ref{lemma:regularization_X} we deduce that one needs to rescale the exponential of the regularized field by a power of $\eps$ in order to obtain a finite limit, giving rise to a GMC measure. 
	\begin{defdef}
		Let $g$ be a smooth Riemannian metric on $\Sigma$ and $\X^g_\eps$ be the regularized field introduced above. We define a Radon measure over $\Sigma$ by setting for $\norm{u}^2<4$
		\begin{equation} \label{eq:GMC_X}
			e^{\ps{u,\X^g(x)}} \d v_g(x) := \lim\limits_{\eps\to 0} \eps^{\frac{|u|^2}{2}} e^{\ps{u,X^g_\eps(x)}} \d v_g(x).
		\end{equation}
	\end{defdef}
	
	Thanks to the theory of GMC, as soon as $\norm{u}^2< 4$ the above convergence stands in probability in the space of Radon measures equipped with the weak topology. Moreover, the limiting measure is finite and non-trivial. Note that as written in the action~\eqref{eq:toda_action}, we will only need to consider some particular vectors $u$ in the direction of the simple roots $(e_i)_{1\leq i\leq r}$, that is $u = \gamma e_i$ for $i=1,\cdots,r$. Then $\norm{u}^2=\gamma^2\ps{e_i,e_i}=:\gamma_i^2$. Since the longest roots are such that $\ps{e_i,e_i}=2$ this explains the constraint that $\gamma$ must belong to $(0,\sqrt{2})$. 
	
	The following Lemma is the companion result of Lemma~\ref{lemma:change of metric GFF} for the GMC:
	\begin{lemma} \label{lemma:conformal change GMC}
		Let $g = e^\varphi g_0$ be a metric in the conformal class of the hyperbolic metric $g_0$, and $u$ in $\a$ with $\norm{u^2}<4$. Then the following equality holds in distribution:
		$$
		e^{\ps{u,\X^{g}}} \d v_{g} \eqlaw e^{(1+\frac{|u|^2}{4})\varphi} e^{\ps{u,\X^{g_0}-m_{g}(\X^{g_0})}} \d v_{g_0}.
		$$
		Hence for any $F$ (resp. $f_i$, $G$) positive continuous over $H^{-1}(\Sigma\to\a,g)$ (resp. $\Sigma$, $\R^r$):
		\begin{eqs} \label{eq:change of metric path integral closed}
			&\int_\a \expect{F\left(\X^{g}+\bm c\right)G\left(\int_\Sigma f_i e^{\ps{\gamma e_i,\X^{g}+\bm c}} \d v_{g}\right)_{1\leq i\leq r}}\d\bm c \\
			&=\int_\a \expect{F\left(\X^{g_0}+\bm c\right)G\left(\int_\Sigma f_i e^{\frac{\ps{Q,\gamma e_i}}{2}\varphi} e^{\ps{\gamma e_i,\X^{g_0}+\bm c}} \d v_{g_0}\right)_{1\leq i\leq r}} \d\bm c.
		\end{eqs}
	\end{lemma}
	\begin{proof}
		Thanks to Lemma~\ref{lemma:regularization_X} together with Lemma~\ref{lemma:change of metric GFF} we deduce that 
		\begin{align*}
			e^{\ps{u,\X^g_\eps(x)}}\d v_g&\eqlaw \lim\limits_{\eps\to 0}\eps^{\frac{\norm{u}^2}2}e^{\ps{u,\mu_{g,\eps}\left[\X^{g_0}(x)-m_g(\X^{g_0})\right](x)}}e^{\varphi(x)}\d v_{g_0}(x)\\
			&=e^{-\ps{u,m_g(\X^{g_0})}}e^{\varphi(x)}e^{\frac{|u|^2}{2}\left(W_{g_0}(x)+\frac12\varphi(x)\right)} \lim\limits_{\eps\to 0}e^{\ps{u,\mu_{g,\eps}\left[\X^{g_0}\right](x)} - \frac12 \E[\ps{u,\mu_{g,\eps}\left[\X^{g_0}\right](x)}^2]}  \d v_{g_0}(x).
		\end{align*}
		We conclude since the limiting GMC measure $\lim\limits_{\eps\to 0}e^{\ps{u,\mu_{g,\eps}\left[\X^{g_0}\right](x)} - \frac12 \E[\ps{u,\mu_{g,\eps}\left[\X^{g_0}\right](x)}^2]}  \d v_{g_0}(x)$ is independent from the regularization scheme (see e.g.~\cite{RV10}), so that we can replace $\mu_{g,\eps}\left[\X^{g_0}\right]$ by $\mu_{g_0,\eps}\left[\X^{g_0}\right]$. The case where $u=\gamma e_i$ is handled by using that $\ps{Q,e_i}=\gamma\frac{\norm{e_i}^2}{2}+\frac2\gamma$ and yields the second item by the change of variable $\bm c\to\bm c+m_{g}(\X^{g_0})$. 
	\end{proof}

	\subsection{Probabilistic interpretation of the path integral}
	We can now give a probabilistic interpretation of the path integral of Toda CFTs~\eqref{eq:toda path integral}. Recall that the background charge is $Q=\gamma\rho+\frac2\gamma\rho^\vee$ and that the cosmological constants $\mu_{B,i}$ are taken positive for $1\leq i\leq r$.
	
	\subsubsection{Definition of the Toda field} We first make sense of the partition function and of the Toda field. This is done as follows:
	\begin{defdef}
		Let $g$ be a smooth Riemannian metric on $\Sigma$ and $F$ any bounded continuous functional over $H^{-1}(\Sigma\to\a,g)$. We set
		\begin{equation} \label{eq:path integral X}
			\ps{F}_g \coloneqq \lim\limits_{\eps\to 0} \ps{F}_{g,\eps},\quad\text{where}\quad\ps{F}_{g,\eps}\coloneqq \mathcal{Z}_\X(\Sigma,g) \int_{\a} \E \left[F(\X_\eps^g + \c) e^{-V_g(\X^g_\eps + \c)}\right] \d\c
		\end{equation}
		with the regularized Toda potential given by
		\begin{equation} \label{eq:reg_action_X}
			V_g(\X^g_\eps + \c) = \frac{1}{4\pi} \int_{\Sigma} R_g \ps{Q,\X^g_\eps + \c} dv_g  + \sum_{i=1}^{r}  \mu_{B,i} e^{\ps{\gamma e_i,\c}} \int_{\Sigma} \eps^{\frac{\gamma_i^2}{2}} e^{\ps{\gamma e_i,\X^g_\eps}} \d v_g.
		\end{equation}
	\end{defdef}
	
	The following statement ensures that the Toda path integral gives rise to a well-defined positive measure on $H^{-1}(\Sigma\to\a,g)$:
	\begin{prop} \label{prop:convergence partition fct X}
		Assume that $\chi(\Sigma)<0$ and take $g$ any smooth Riemannian metric on $\Sigma$. Let $F$ be any continuous, bounded functional over $H^{-1}(\Sigma\to\a,g)$. Then the limit~\eqref{eq:path integral X} exists and belongs to $(0,\infty)$. Moreover the limiting partition function satisfies a Weyl anomaly of the form~\eqref{eq:weyl} and is invariant under orientation-preserving diffeomorphisms~\eqref{eq:diffeo invariance}.
	\end{prop}
	See Theorems~\ref{thm:existence} and~\ref{thm:axioms} for a more precise statement. We defer the proof of this statement to Appendix~\ref{sec:appendix_proof}. As an immediate consequence of this result we have the following:
	\begin{coro}
		The assignment of $F \mapsto \ps{F}_g$ for $F$ a bounded continuous functional over $H^{-1}(\Sigma\to\a,g)$ defines a unique Borel measure over $H^{-1}(\Sigma\to\a,g)$.
	\end{coro}
	\begin{rem}
		This naturally defines a probability measure on $H^{-1}(\Sigma\to\a,g)$ equipped with its Borel $\sigma$-algebra by considering the mapping $F \mapsto \frac{\ps{F}_g}{\ps{1}_g}$.
	\end{rem}

	\begin{rem}
		In the case where $\chi(\Sigma) \ge 0$ the partition function is infinite because of the behavior of the integral in the zero-mode $\bm c$ when $\ps{\c,e_i} \to -\infty$. We can still define the measure by $F \mapsto \ps{F}_g$ for suitable functionals, although it will have infinite total mass.
	\end{rem}

	\subsubsection{Correlation functions of Vertex Operators}
	Key observables in Toda theory are the Vertex Operators, formal exponentials of the random distribution $\Phi$ under the law described by the path integral \eqref{eq:toda path integral}: $\ps{e^{\ps{\alpha, \Phi(z)}}}_g$ for some $\alpha\in\a$ and $z\in\Sigma$. Averaging products of Vertex Operators with respect to the law defined by the path integral amounts to defining their correlation functions: we study here the existence and some properties of such objects.
	
	\begin{defdef}
		Consider a family of $N \in \N$ pairs $(x_k,\alpha_k)_{k=1,...,N}$ with $x_k \in \Sigma$ all distinct and $\alpha_k \in \a$. The regularized Vertex Operator at $x_k$ with weight $\alpha_k$ is defined by
		\begin{equation} \label{eq:reg VO definition bulk X}
			V_{\alpha_k,\eps}(x_k) \coloneqq \eps^{\frac{|\alpha_k|^2}{2}} e^{\ps{\alpha_k,\X_\eps^g(x_k) + \c}}.
		\end{equation}
		The correlation function of Vertex Operators is then, if it makes sense, the limit
		\begin{equation} \label{eq:limit correl X}
			\begin{split}
				& \ps{\prod_{k=1}^NV_{\alpha_k}(x_k)}_g \coloneqq \lim\limits_{\eps\to 0} \ps{\prod_{k=1}^NV_{\alpha_k,\eps}(x_k)}_{g,\eps},\quad\text{where}\\
				&\ps{\prod_{k=1}^NV_{\alpha_k,\eps}(x_k)}_{g,\eps} = \mathcal{Z}_\X(\Sigma,g) \int_{\a} \E \left[\ps{\prod_{k=1}^NV_{\alpha_k,\eps}(x_k)}_{g,\eps}e^{-V_g(\X^g_\eps + \c)}\right] \d\c.
			\end{split}
		\end{equation}
	\end{defdef}   
	We prove in Appendix~\ref{sec:appendix_proof} that it is indeed the case under the Seiberg bounds:
	\begin{thm}\label{thm:seiberg}
		For $g$ a smooth Riemannian metric on $\Sigma$, the limit \eqref{eq:limit correl X} exists and belongs to $(0,\infty)$ if and only if the Seiberg bounds~\eqref{eq:Seiberg} hold. The limiting correlation function verifies the Weyl anomaly~\eqref{eq:weyl} and is invariant under orientation-preserving diffeormorphisms~\eqref{eq:diffeo invariance}.
	\end{thm}
	
	\begin{rem}
		If the Seiberg bounds hold the map $F \mapsto \ps{F\prod_{k=1}^NV_{\alpha_k}(x_k)}_g/\ps{\prod_{k=1}^NV_{\alpha_k}(x_k)}_g$ defines a probability measure on $H^{-1}(\Sigma\to\a,g)$ equipped with its Borel $\sigma$-algebra.
	\end{rem}
	
	\section{Boundary Toda theories}\label{sec:bdry}
	Having defined Toda CFTs when the underlying surface is a closed hyperbolic Riemann surface, we now turn to the case where the compact surface $\Sigma$ has a non-empty boundary. As explained in the introduction a major difference is that the presence of a boundary implies that one needs to specify boundary conditions for the field $\bphi$ in the Toda action, thus giving rise to several possible constructions for the theory with a boundary. These boundary conditions correspond to automorphisms of the underlying root system $\mc R$.
	
	In this section we assume that $\Sigma$ has a non-empty boundary, is of hyperbolic type and is equipped with a smooth Riemannian metric $g$. We also fix an element $\tau$ of Out($\mc R$): as explained in Subsections~\ref{subsec:out} and~\ref{subsec:fold} this allows one to write $\a$ as the orthogonal sum $\a=\a_N\bigoplus\a_D$ with 
	the corresponding orthogonal projections denoted $p_N$ and $p_D$. We will also frequently use the notation $u^*$ for $\tau(u)$ where $u$ is in $\a$.
	
	\subsection{The Gaussian Free Fields considered}
	
	\subsubsection{Neumann and Dirichlet Gaussian Free Fields}
	In the same way as in Section~\ref{sec:GFF closed}, it is possible to define a Gaussian field with covariance structure given by either the Neumann or Dirichlet Green's function introduced respectively in Equations~\eqref{eq:Neumann_pb_green} and~\eqref{eq:dirich_pb_green}. 
	\begin{defdef} \label{def:neumann and dirichlet gff}
		The (vectorial) Neumann GFF $\X^{g,N}$ is the centered Gaussian field with covariance structure given for all $u,v\in\a$ and all $x\neq y\in\overline{\Sigma}$ by
		$$
		\E[\ps{u,\X^{g,N}(x)}\ps{v,\X^{g,N}(y)}] = \ps{u,v}G^{N}_g(x,y)
		$$
		where we used the same abuse of notation as when defining the closed GFF. Likewise the Dirichlet GFF $\X^{g,D}$ is centered with covariance  for all $u,v\in\a$ and all $x\neq y\in\overline{\Sigma}$:
		$$
		\E[\ps{u,\X^{g,D}(x)}\ps{v,\X^{g,D}(y)}] = \ps{u,v}G^{D}_g(x,y).
		$$
	\end{defdef}
	Due to the presence of the zero-mode, to make sense of the path integral the law of the Neumann GFF has to be tensorized with the Lebesgue measure just as in the closed case. We will take this into account when defining the partition function. 
	
	Like before we can see these fields as vectorial GFFs, e.g. $\X^{g,N} = \sum_{i=1}^r X^{g,N}_i\bm v_i$ where $(\bm v_i)_{1\leq i\leq r}$ is an orthonormal basis of $\a$ and with the $X^{g,N}_i$ independent 1-dimensional Neumann GFFs on $\Sigma$ (and likewise for $\X^{g,D}$). This allows to define their partition function:
	\begin{defdef}
		Let $g$ be a smooth Riemannian metric over $\Sigma$, $\Delta^N_g$ be the one-dimensional Neumann Laplacian on $\Sigma$ and $\Delta_g^D$ be the one-dimensional Dirichlet Laplacian on $\Sigma$. We set:
		\begin{equation} \label{eq:partition function Y}
			\mathcal{Z}_N(\Sigma,g) := \left(\frac{\det{'}\left(-\frac{1}{2\pi}\Delta_g^N\right)}{v_g(\Sigma)}\right)^{-\frac r2}\exp\left(-\frac{r}{8\pi}\int_{\partial\Sigma}k_g \d \lambda_g\right).
		\end{equation}
		Likewise the partition function of the Dirichlet field is given by a regularized determinant: 
		\begin{equation}
			\mathcal{Z}_{D}(\Sigma,g) \coloneqq \det\left(-\frac{1}{2\pi}\Delta_g^D\right)^{-\frac r2}\exp\left(\frac{r}{8\pi} \int_{\partial\Sigma} k_g \d\lambda_g\right).
		\end{equation}
	\end{defdef}
	The terms $\exp\left(\pm\frac{r}{8\pi}\int_{\partial\Sigma}k_g \d \lambda_g\right)$ are needed for the conformal anomaly (see Proposition~\ref{prop:conformal anomaly boundary}). 
	
	\subsubsection{Gaussian Free Field with Cardy boundary conditions}
	
	Lemma~\ref{lemma:decomposition_green_function} suggests an alternative way of defining the Neumann and Dirichlet GFFs. Indeed, let $g_0$ be a uniform metric of type 1 on $\Sigma$: one can double the surface $\Sigma$ to obtain a closed surface $\hat{\Sigma}$ equipped with a smooth metric $\hat{g_0}$. Now let $\X^{\hat{g_0}}$ be the vectorial GFF on $\hat{\Sigma}$. Then the fields
	\begin{equation} \label{eq:2D Neumann GFF}
		\X^{g_0,N} \coloneqq \frac{\X^{\hat{g_0}} + \X^{\hat{g_0}}\circ\sigma}{\sqrt{2}} \qt{and} \X^{g_0,D} \coloneqq \frac{\X^{\hat{g_0}} - \X^{\hat{g_0}}\circ\sigma}{\sqrt{2}}
	\end{equation}
	(where $\sigma$ is the canonical involution) have respectively the law of the Neumann and Dirichlet GFFs on $\Sigma$ (in virtue of Lemma~\ref{lemma:decomposition_green_function}). Based on this observation we wish to define a new field $\X^{g_0,C}$ over $\Sigma$ based on the doubling trick discussed in the introduction. Namely $\X^{g_0,C}$ would be such that $\X^{\hat g_0,C}$ and $(\X^{\hat g_0,C})^*\circ\sigma$ have the same law, where $\cdot^*$ is an order two outer automorphism of the underlying root system and $\X^{\hat g_0,C}$ is the extension of $\X^{g_0,C}$ to $\hat\Sigma$ defined using $\sigma$. To this end in the metric $g_0$ we define a new field on $\bar{\Sigma}$ via
	\begin{equation} \label{eq:2D Cardy GFF}
		\X^{g_0,C} = \frac{\X^{\hat{g_0}} + (\X^{\hat{g_0}})^*\circ\sigma}{\sqrt{2}}\cdot
	\end{equation}
	Some key remarks have to be made at this stage:
	\begin{enumerate}
		\item a straightforward consequence of Equation~\eqref{eq:2D Cardy GFF} is that $\X^{\hat{g_0},C}=(\X^{\hat{g_0},C})^*\circ\sigma$;
		\item the field $\X^{g_0,C}$ has the following covariance structure for $x,y\in\bar{\Sigma}$ and $u,v\in\a$:
		\begin{equation} \label{eq: cov structure Z}
			\E[\ps{u,\X^{g_0,C}(x)}\ps{v,\X^{g_0,C}(y)}] = \ps{u,v}G_{\hat{g_0}}(x,y) + \ps{u,v^*}G_{\hat{g_0}}(x,\sigma(y));
		\end{equation}
		\item let us decompose $\X^{g_0,C}$ in $\a=\a_N\bigoplus\a_D$ (recall Subsection~\ref{subsec:out}) by writing
		\begin{equation} \label{eq:decomposition Cardy GFF}
			\X^{g_0,C}=\X^{g_0,N}+\X^{g_0,D},\quad\X^{g_0,N}\coloneqq\frac{\X^{g_0,C}+(\X^{g_0,C})^*}{2}\quad\text{and}\quad \X^{g_0,D}\coloneqq\frac{\X^{g_0,C}-(\X^{g_0,C})^*}{2}\cdot
		\end{equation}
		Then $\X^{g_0,N}$ (resp. $\X^{g_0,D}$) has the law of a Neumann (resp. Dirichlet) GFF from $\Sigma$ to $\a_N$ (resp. $\a_D$). The two fields are independent. Besides, we can rewrite \eqref{eq: cov structure Z} as 
		\begin{equation} 
			\E[\ps{u,\X^{g_0,C}(x)}\ps{v,\X^{g_0,C}(y)}] = \ps{u,\frac{v+v^*}2}G_{\hat{g_0}}^N(x,y) + \ps{u,\frac{v-v^*}2}G_{\hat{g_0}}^D(x,\sigma(y)).
		\end{equation}
	\end{enumerate}
	This construction is limited to the uniform metric $g_0$ and to outer automorphisms of order two. However, these observations suggest the following definition for the Cardy GFF:
	\begin{defdef}
		Given a smooth Riemannian metric $g$ on $\Sigma$ and an automorphism $\tau$ of $\mc R$, the (vectorial) Cardy GFF $\X^{g,C}$ is the centered Gaussian field with covariance structure given for all $u,v\in\a$ and all $x\neq y\in\overline{\Sigma}$ by
		\begin{equation}\label{eq:GFF_Cardy}
			\E[\ps{u,\X^{g,C}(x)}\ps{v,\X^{g,C}(y)}] = \ps{p_Nu,p_Nv}G^{N}_g(x,y) + \ps{p_Du,p_Dv}G^D_g(x,y).
		\end{equation}
	\end{defdef}     
	For such a GFF $\X^{g,C}$ we have a decomposition similar to Equation~\eqref{eq:decomposition Cardy GFF}, which allows to define its partition function: set $d_N\coloneqq \dim(\a_N)$ and $d_D\coloneqq \dim(\a_D)$ so that in particular $d_N+d_D=r$ --- recall that the values of these quantities are explicit and given in Subsection~\ref{subsec:out}. 
	\begin{defdef}
		Let $g$ be a smooth Riemannian metric on $\Sigma$. We set
		\begin{equation}
			\mathcal{Z}_{C}(\Sigma,g) \coloneqq  \left(\frac{\det{'}\left(-\frac{1}{2\pi}\Delta_g^N\right)}{v_g(\Sigma)}\right)^{-\frac{d_N}2}\det\left(-\frac{1}{2\pi}\Delta_g^D\right)^{-\frac{d_D}2} \exp\left(\frac{d_D-d_N}{8\pi} \int_{\partial\Sigma} k_g \d\lambda_g\right).
		\end{equation}
	\end{defdef}
	This formula agrees with~\eqref{eq:partition function Y} when the outer automorphism $\tau$ is the identity (for which $d_N=r$ and $d_D=0$). Via Lemma~\ref{lemma:background/polyakov formula} it satisfies the following conformal anomaly formula:
	\begin{prop}\label{prop:conformal anomaly boundary}
		Let $g' = e^\varphi g$ be a smooth Riemannian metric in the conformal class of $g$. Then we have
		\begin{equation} 
			\mathcal{Z}_{C}(\Sigma,g') = \mathcal{Z}_{C}(\Sigma,g) \exp\left(\frac{r}{96\pi} \left( \int_{\Sigma} (|d_g\varphi|_g^2 + 2R_g\varphi) \d v_g + 4\int_{\partial\Sigma} k_g \varphi \d\lambda_g\right)\right).
		\end{equation}
	\end{prop}
	\begin{proof} 
		Thanks to Lemma \ref{lemma:background/polyakov formula} we have 
		\begin{align*}
			&\log \mathcal{Z}_{C}(\Sigma,g') - \log \mathcal{Z}_{C}(\Sigma,g) = \frac{r}{96\pi} \left( \int_{\Sigma} (|d_g\varphi|_g^2 + 2R_g\varphi) \d v_g + 4\int_{\partial\Sigma} k_g \varphi \d\lambda_g\right) + R\\
			&\text{with}\quad R = \frac{d_N-d_D}{16\pi} \left( \int_{\partial\Sigma} \neu \varphi \d\lambda_g + 2\int_{\partial\Sigma} k_{g} \d\lambda_{g} - 2\int_{\partial\Sigma} k_{g'} \d\lambda_{g'} \right)
		\end{align*}
		which is seen to vanish thanks to \eqref{eq:conf change curvature}.
	\end{proof}
	
	\subsection{Gaussian Multiplicative Chaos}
	In order to make sense of the exponential of the GFF we need to regularize the field $\X^{g,C}$. To this end for $x\in\overline\Sigma$ and $\eps>0$ small let $C_g(x,\eps)\coloneqq\left\{y\in\overline{\Sigma}, d_g(x,y)=\eps\right\}$ be the geodesic (semi-)circle of radius $\eps>0$ centered at $x$, and $\mu_{g,\eps}[\cdot](x)$ be the uniform measure probability on $C_{g}(x,\eps)$. 
	\begin{defdef}
		The regularized GFF $\X^{g,C}_\eps$ is the average of $\X^{g,C}$ over geodesic circles:
		\begin{equation}
			\X^{g,C}_\eps(x)\coloneqq\mu_{g,\eps}[\X^{g,C}](x).
		\end{equation}
	\end{defdef}
	The averaged GFF $\X^{g,C}_\eps$ satisfies:
	\begin{lemma}\label{lemma:regularization GFF boundary case cardy}
		The field $\X^{g,C}_\eps$ has covariance kernel given by
		\begin{equation}\label{eq:cov_Zeps}
			\E[\ps{u,\X^{g,C}_\eps(x)}\ps{v,\X^{g,C}_\eps(y)}] = \iint (\ps{p_Nu,p_Nv} G_g^N + \ps{p_Du,p_Dv} G_g^D) \d\mu_{g,\eps}(x) \d\mu_{g,\eps}(y).
		\end{equation}
		Besides, if $g = e^\varphi g_0$ is a smooth Riemannian metric in the conformal class of the uniform type 1 metric $g_0$, then as $\eps\to 0$ and uniformly in $x\notin \partial \Sigma$:
		\begin{equation}\label{eq:variance_Zeps}
			\E[\ps{u,\mu_{g,\eps}[\X^{g_0,C}](x)}^2] = \left\lbrace \begin{array}{ll}
				|u|^2\left(-\log \eps + W_{g_0}(x)+\frac12\varphi(x)\right) + \ps{u,p_Nu-p_Du}G_{g_0}(x,\sigma(x)) + o(1)\\
				2\norm{p_Nu}^2\left(-\log\eps + W_{g_0}(x)+\frac12\varphi(x)\right) + o(1)
			\end{array} \right. 
		\end{equation}
		with $W_{g_0}$ a smooth function on $\bar\Sigma$, $x\in\Sigma$ on the first line and $x\in\partial\Sigma$ on the second one.
	\end{lemma}
	\begin{proof}
		This again follows from the computations conducted in~\cite[Lemma 3.2 and Equation (3.10)]{GRV19} by writing $G_{g_0}^{N/D}$ in terms of $G_{\hat g_0}$ as in Lemma~\ref{lemma:decomposition_green_function}. See also the proof of Lemma~\ref{lemma:regularization_X}.
	\end{proof}
	\begin{defdef}
		The GMC measures associated to $\X^{g,C}$ are then defined by the limits:
		\begin{align} \label{eq:GMC_bulk_Z}
			e^{\ps{u,\X^{g,C}(x)}} \d v_g(x) &:= \lim\limits_{\eps\to 0} \eps^{\frac{|u|^2}{2}} e^{\ps{u,\X^{g,C}_\eps(x)}} \d v_g(x)
		\end{align}
		for $x$ in the bulk, while for $x\in\partial\Sigma$ they are given by
		\begin{align} \label{eq:GMC_bdry_Z}
			e^{\ps{u,\X^{g,C}(x)}} \d \lambda_g(x) &:= \lim\limits_{\eps\to 0} \eps^{\norm{p_Nu}^2} e^{\ps{u,\X^{g,C}_\eps(x)}} \d \lambda_g(x).
		\end{align}
	\end{defdef}
	Like before all the convergences stand in probability in the space of Radon measures equipped with the weak topology, provided $\norm{u}^2<4$ for the bulk measure and $\norm{p_Nu}^2<1$ for the boundary one. We stress that in the case where $\tau$ is the identity, the above statement amounts to defining the vectorial Neumann GMC on $\Sigma$.
	
	Finally we provide the analogue of Lemma~\ref{lemma:conformal change GMC} for the behavior of the GMC measures under a conformal change of metric. Recall that the $f_j$ are a basis of $\a_N$ of the form $f_j=p_N(e_i)$.
	\begin{lemma} \label{lemma:conformal change GMC cardy} 
		Let $g = e^\varphi g_0$ be a metric in the conformal class of the uniform type 1 metric $g_0$, and $u$ in $\a$. Then the following equalities hold in distribution:
		\begin{align*}
			&e^{\ps{u,\X^{g,C}}} \d v_{g} \eqlaw e^{\left(1+\frac{|u|^2}{4}\right)\varphi} e^{\ps{u,\X^{g_0,C}-m_g(\X^{g_0,C})}} \d v_{g_0}\\
			\qt{and}  	&e^{\ps{u,\X^{g,C}}} \d \lambda_{g} \eqlaw e^{\left(1+\norm{p_Nu}^2\right)\frac\varphi2} e^{\ps{u,\X^{g_0,C}-m_g(\X^{g_0,C})}} \d \lambda_{g_0}.
		\end{align*}
	\end{lemma}
	\begin{proof}
		It is exactly the same proof as in the closed case, using Lemma~\ref{lemma:regularization GFF boundary case cardy}. Namely for $x$ in the bulk
		\begin{align*}
			&e^{\ps{u,\X^{g,C}(x)}} \d v_g(x)\eqlaw \lim\limits_{\eps\to 0} \eps^{\frac{|u|^2}{2}} e^{\ps{u,\X^{g_0,C}_\eps(x)-m_g(\X^{g_0,C})}} e^{\varphi(x)} \d v_{g_0}(x) \\
			& = e^{\left(1+\frac{|u|^2}{4}\right)\varphi(x)} e^{-\ps{u,m_g(\X^{g_0,C})}} F_{\hat g_0}(x)\lim\limits_{\eps\to 0} e^{\ps{u,\mu_{g,\eps}(\X^{g_0,C})(x)} - \frac12 \E[\ps{u,\mu_{g,\eps}(\X^{g_0,C})(x)}^2]} \d v_{g_0}(x) 
		\end{align*}
		where $F_{\hat g_0}(x)=e^{\frac{1}{2}\left(|u|^2W_{\hat{g_0}}(x) + \ps{u,p_Nu-p_Du} G_{\hat{g_0}}(x,\sigma(x))\right)}$ is independent of $\varphi$. For $x\in\partial\Sigma$
		\begin{align*}
			&e^{\ps{u,\X^{g,C}(x)}} \d \lambda_g(x) \eqlaw \lim\limits_{\eps\to 0} \eps^{\norm{p_Nu}^2}e^{\ps{u,\mu_{g,\eps}[\X^{g_0,C}](x)-m_g(\X^{g_0,C})}} e^{\frac{\varphi(x)}{2}} \d \lambda_{g_0}(x) \\
			& = e^{\left(1+\norm{p_Nu}^2\right)\frac{\varphi(x)}2}e^{-\ps{u,m_g(\X^{g_0,C})}} e^{\norm{p_Nu}^2W_{\hat{g_0}}(x)}\times\lim\limits_{\eps\to 0} e^{\ps{u,\mu_{g,\eps}(\X^{g_0,C})(x)} - \frac12 \E[\ps{u,\mu_{g,\eps}(\X^{g_0,C})(x)}^2]} \d \lambda_{g_0}(x). 
		\end{align*}
		We end the proof as in Lemma~\ref{lemma:conformal change GMC}, using the fact that the limiting GMC measures are independent of the regularization scheme. 
	\end{proof}
	
	\subsection{Different boundary models}
	
	\subsubsection{The Toda action and invariant subspaces}
	
	Thanks to Equation~\eqref{eq:decomposition Cardy GFF}, for $\tau\in\text{Out}(\mc R)$ we can decompose $\X^{g,C}$ as a sum of a Neumann GFF $\X^{g,N}$ and a Dirichlet GFF $\X^{g,D}$. Like in the closed case the Neumann GFF is defined up to an additive constant (fixed by the requirement that $m_g(\X^{g,N})=0$), but this is not the case for the Dirichlet GFF. As a consequence the space where one should take the constant mode $\bm c$ is not $\a$ but rather $\a_N$.
	Likewise functionals of the field should take into account its specific form. In particular for exponentials of the Toda field $e^{\ps{\beta,\bphi(s)}}$, when  $s\in\partial\Sigma$ we can assume that $\beta\in\a_N$ since $\bphi(s)\in\a_N$ which is orthogonal to $\a_D$. 
	
	\begin{defdef}\label{def:part_fct_bord}
		Given cosmological constants $\mu_{B,i}>0$ and $\mu_j\ge 0$ for $1\leq i\leq r$ and $1\leq j\leq d_N$, and $g$ a smooth Riemannian metric on $\Sigma$, the path integral is defined by setting, for $F$ a bounded continuous functional on $H^{-1}(\Sigma\to\a,g)$, 
		\begin{equation} \label{eq:path integral Y}
			\ps{F}_{g,\tau} \coloneqq \lim\limits_{\eps\to 0} \ps{F}_{g,\tau,\eps},\quad  \ps{F}_{g,\tau,\eps} \coloneqq \mathcal{Z}_C(\Sigma,g) \int_{\a_N} \E \left[F(\X^{g,C}_\eps + \c) e^{-V_g(\X^{g,C}_\eps + \c)}\right] \d\c.
		\end{equation}
		Here the regularized potential is given by
		\begin{equation} \label{eq:reg_action_Y}
			\begin{split}
				V_g(\X^{g,C}_\eps + \c) = \frac{1}{4\pi} \int_{\Sigma} R_g \ps{Q,\X^{g,C}_\eps + \c} \d v_g + \frac{1}{2\pi} \int_{\partial\Sigma} k_g \ps{Q,\X^{g,C}_\eps + \c} \d\lambda_g \\ + \sum_{i=1}^{r} \mu_{B,i} e^{\ps{\gamma e_i,\c}} \int_\Sigma \eps^{\frac{\gamma_i^2}{2}} e^{\ps{\gamma e_i,\X^{g,C}_\eps}} \d v_g + \sum_{j=1}^{d_N} e^{\ps{\frac{\gamma}{2}f_j,\c}} \int_{\partial\Sigma} \eps^{\frac{\gamma^2\norm{ f_j}^2}{4}} e^{\ps{\frac{\gamma}{2} f_j,\X^{g,C}_\eps}} \mu_j (\d \lambda_g) .
			\end{split}
		\end{equation}
	\end{defdef}
	where $\mu_j(\d \lambda_j) \coloneqq \mu_j\d\lambda_j$ (the reason for this notation will become clear later). We then have the following analogue of Proposition~\ref{prop:convergence partition fct X}, proved in Appendix~\ref{sec:appendix_proof}.
	\begin{prop} \label{prop:part_fct_bord}
		In the setting of Defintion~\ref{def:part_fct_bord}, assume that $\chi(\Sigma)<0$. Then the limit~\eqref{eq:path integral Y} exists and belongs to $(0,\infty)$.
	\end{prop}
	
	\subsubsection{Around the Weyl anomaly}\label{subsec:different_models}
	Having defined the path integral, we now would like to deduce from Lemma~\ref{lemma:conformal change GMC cardy} the analogue of the change of metric formula~\eqref{eq:change of metric path integral closed} for the path integral in the boundary case, which in turn allows to prove a Weyl anomaly for the correlation functions. However, there is a major obstruction that comes from the fact that, when $\tau\neq\Id$, the folded root system $(f_j)_{j=1,...,d_N}$ is \emph{not} simply-laced, so that the conformal anomaly of the boundary GMC measures associated with short roots is expressed in terms of a different background charge $Q_\tau$ \eqref{eq:Q tau gamma tau}, thus breaking the conformal symmetry of the model. We thus need to distinguish between the cases where $\tau=\Id$ or not. 
	
	\textbf{Neumann boundary conditions.} When the outer automorphism under consideration is the identity the construction leads to Neumann (or free) boundary conditions. In this case the model is conformally covariant (Theorem~\ref{thm:axioms}), and its algebra of symmetry is given by the $W$-algebra constructed from $\g$, both in the bulk and on the boundary, as shown in~\cite{CH_sym1}.
	
	\textbf{Cardy boundary conditions.} In view of Lemma~\ref{lemma:conformal change GMC cardy} we know that the conformal anomaly for the boundary GMC measure is given by $\exp\left(\frac{\ps{Q_\tau,\gamma f_j}}{4}\varphi\right)$ where $Q_\tau$ is defined in Equation~\eqref{eq:Q tau gamma tau}. It suggests that one should consider a model for which the background charge on the boundary is set to be $Q_\tau$ instead of $Q$. However, this spoils the construction in the bulk by breaking conformal covariance there! A way to overcome this issue is to assume that $\mu_{j}=0$ for all $j$ such that $\ps{Q, f_j}\neq \ps{Q_\tau, f_j}$, which is the case when $j\in\{1,\cdots,d_N\}\setminus I_\tau$.
	
	The  model then enjoys conformal Weyl anomaly and diffeomorphism invariance:
	\begin{prop} \label{prop:convergence partition function bdry}
		In the setting of Proposition~\ref{prop:convergence partition function bdry}, assume that $\mu_j\neq0$ only if $j\in I_\tau$. Then the model satisfies the Weyl anomaly formula~\eqref{eq:weyl} and diffeomorphism invariance~\eqref{eq:diffeo invariance}. 
	\end{prop}
	
	\subsubsection{Correlation functions}
	
	On the surface with boundary $\Sigma$ we consider not only $N$ insertion points in the bulk together with their associated weights $(x_k,\alpha_k)_{k=1,...,N}$ but also $M \in \N$ insertions on the boundary $(s_l,\beta_l)_{l=1,...,M}$ with $s_l \in \partial\Sigma$ and $\beta_l\in\a_N$. 
	
	In the boundary case one can slightly generalize the model by introducing different cosmological constants between the boundary insertion points. Namely, suppose that the boundary of $\Sigma$ is composed of $\k$ connected components $\mc C_n$, $n=1,...,\k$, and that $M_n$ insertions lie on the part $\mc C_n$ of $\partial\Sigma$ (thus $M_1+...+M_\k=M$). Let us introduce a family of non-negative\footnote{In fact, one can even choose the boundary cosmological constants in $\C$ with $\Re(\mu_{j,l_n}^{(n)}) \ge 0$.} cosmological constants $(\mu_{j,l_n}^{(n)})$ with $j=1,\cdots,d_N$, $n=1,...,\k$ and $l_n=1,...,M_n$. Then in the action \eqref{eq:reg_action_Y} we set $\mu_j(\d \lambda_g) = \sum_{n=1}^\k \sum_{l_n=1}^{M_n} \mu_{j,l_n}^{(n)} \ind_{(s_{l_n}^{(n)},s_{l_n+1}^{(n)})} \d \lambda_g$, $(s_{l_n}^{(n)},s_{l_n+1}^{(n)})$ being the part of $\partial\Sigma$ between the two insertion points $s_{l_n}^{(n)}$ and $s_{l_n+1}^{(n)}$ and with the convention  that $s_{M_n+1}^{(n)} = s_1^{(n)}$.
	
	\begin{defdef}
		The bulk Vertex Operators are defined similarly as in Equation~\eqref{eq:reg VO definition bulk X} by
		\begin{equation} \label{eq:reg VO definition bulk boundary case}
			V_{\alpha_k,\eps}(x_k) = \eps^{\frac{|\alpha_k|^2}{2}} e^{\ps{\alpha_k,\X^{g,C}_\eps(x_k) + \c}}\qt{and}V_{\beta_l,\eps}(s_l) = \eps^{\frac{|\beta_l|^2}{4}} e^{\ps{\frac{\beta_l}{2},\X^{g,C}_\eps(s_l) + \c}}
		\end{equation}
		on the boundary. We then define the correlation function in the boundary case by the limit
		\begin{equation} \label{eq:correl boundary}
			\begin{split}
				\ps{\prod_{k=1}^NV_{\alpha_k}(x_k)\prod_{l=1}^MV_{\beta_l}(s_l)}_{g,\tau} &\coloneqq \lim\limits_{\eps\to 0}\text{ } \ps{\prod_{k=1}^NV_{\alpha_k,\eps}(x_k)\prod_{l=1}^MV_{\beta_l,\eps}(s_l)}_{g,\tau,\eps}\\
				\coloneqq\lim\limits_{\eps\to 0}\text{ } &\mathcal{Z}_C(\Sigma,g) \int_{\a_N} \E \left[\prod_{k=1}^NV_{\alpha_k,\eps}(x_k)\prod_{l=1}^MV_{\beta_l,\eps}(s_l)e^{-V_g(\X^{g,C}_\eps + \c)}\right] \d\c.
			\end{split}
		\end{equation}
	\end{defdef}
	Theorems~\ref{thm:existence} and \ref{thm:axioms} (proved in Appendix~\ref{sec:appendix_proof}) state that the limit is well-defined and non-trivial if and only if the Seiberg bounds hold true, and that the limiting correlation functions do enjoy diffeomorphism invariance, and satisfy the Weyl anomaly formula, under the assumptions of Proposition~\ref{prop:convergence partition function bdry}.
	
	\subsection{Algebra of symmetry}\label{subsec:W_twist}
	We briefly discuss here what is the algebra of symmetry of the models constructed, especially when $\tau$ is non-trivial. To this end let us assume that $\Sigma$ is the upper-half plane $\H$ equipped with a metric conformally equivalent to $g=\frac{\norm{dz}^2}{(1+\norm{z}^2)^2}$. This discussion strongly relies on the framework developed in~\cite{Cer_VOA}, to which we refer for more details, and definitions of the objects involved and of the notations used.
	
	\subsubsection{From the Heisenberg vertex algebras to $W$-algebras}
	Following~\cite[Section 3]{Cer_VOA}, inside $\H$ the symmetry algebra associated with the free-field $\X^g$ is, independently of the boundary conditions chosen, described by the Heisenberg vertex algebra of rank $r$. It is constructed out of a Fock space $\V_{+,\bm c}$ of rank $r$, generated by acting on a \textit{vacuum state} $\vac$ . Compared to the free-field theory, the addition of the Toda potential $e^{\ps{\gamma e_i,\X}}$ to the action breaks this symmetry but the $W$-symmetry, encoded by $W$-algebras, is preserved~\cite[Theorem 5.5]{Cer_VOA}. This is due to the definition of the $W$-algebra associated to $\g$ via the intersection of kernels of screening operators $\left(Q_i^+\right)_{1\leq i\leq r}$,  formally defined by $Q_i^+\coloneqq \oint\mc V_{\gamma e_i}^+(z)\d z$ where $\mc V_{\gamma e_i}$ is a bosonic Vertex Operator, see~\cite[Subsection 4.2]{Cer_VOA}. Namely, the restriction of the vertex algebra structure from $\V_{+,\bm c}$ to $\MW{\g}\coloneqq\bigcap_{1\leq i\leq r}\text{Ker}_{\V_{+,\bm c}}\left(Q_{i}^+\right)$ defines the $W$-algebra associated with $\g$. For generic values of $\gamma$, this $W$-algebra is generated using the modes of currents $\left(\Wb^{(s_i)}[\Psi](z)\right)_{1\leq i\leq r}$ of spins $s_i$, where $s_i-1$ ranges over the exponents of $\g$. These currents may be thought of as polynomials in the holomorphic derivatives of the Toda field $\Phi$, or rather its algebraic counterpart $\Psi$ (denoted $\Phi$ in~\cite[Equation (3.14)]{Cer_VOA}) which is such that $\partial\Psi(z)=\sum_{n\in\Z}\A_nz^{-n-1}\in\text{End}(\V_{+,\bm c})[[z,z^{-1}]]$.
	In agreement with~\cite[Theorem 4.6.9]{FF_QG} and~\cite[Theorem 4.3]{Cer_VOA}, these currents can be chosen in such a way that their modes $(\Wb^{(s_i)}_{-n_i})_{\substack{1\leq i\leq r\\ n_i\geq s_i}}$ freely generate $\MW{\g}$, with in addition the commutation relations $[\Wb^{(2)}_{n},\Wb^{(s_i)}_{m}]=\left(n(s_i-1)-m\right)\Wb^{(s_i)}_{m}$ for all $1\leq i\leq r$ and $n,m\in\mathbb Z$. Here $\Wb^{(2)}$ is the stress-energy tensor $\Wb^{(2)}[\Psi]=\ps{Q,\partial^2\Psi}-\ps{\partial\Psi,\partial\Psi}$. 
	
	\subsubsection{Covariance of the $W$-algebra currents under $\tau\in\text{Out}(\mc R)$}
	Compared to the bulk case, in the presence of a boundary the choice of an element $\tau$ in $\text{Out}(\mc R)$ has concrete implications on the properties of the symmetry algebra. Namely, when $\tau$ is of order two, in view of Equation~\eqref{eq:2D Cardy GFF}, one can still make sense of $\X^{g,C}(\bar z)$ for $z\in\H$, which is nothing but $\tau(\X^{g,C})=(\X^{g,C})^*$. As such, with a slight abuse of notation, we have $\Wb^{(s_i)}[\Psi](\bar z)=\Wb^{(s_i)}[\Psi^*](z)$. Now, and up to redefining the currents $\Wb^{(s_i)}$, we can show that the effect of changing $\Psi$ to $\Psi^*$ has an explicit effect on these currents as follows:
	\begin{prop}\label{prop:W-alg}
		Assume that $\g\neq \mathfrak{so}_{4n}$ for some $n\geq2$. There exist currents $\Wb^{(s_i)}[\Psi]$ of spins $s_i$, with $s_i-1$ ranging over the exponents of $\g$, such that for generic values of $\gamma$:
		\begin{itemize}
			\item $\MW{\g}$ is freely generated by acting on $\vac$ with the modes $(\Wb^{(s_i)}_{-n_i})_{\substack{1\leq i\leq r\\ n_i\geq s_i}}$;
			\item $[\Wb^{(2)}_{n},\Wb^{(s_i)}_{m}]=\left(n(s_i-1)-m\right)\Wb^{(s_i)}_{m}$ for all $1\leq i\leq r$ and $n,m\in\mathbb Z$;
			\item for $\tau\in\text{Out}(\mc R)$ of order $2$ we have $\Wb^{(s_i)}[\tau\Psi]=\pm\Wb^{(s_i)}[\Psi]$.
		\end{itemize}
        In the case of $\g=\sl_n$ we have $\Wb^{(s_i)}[\tau\Psi]=\left(-1\right)^{s_i}\Wb^{(s_i)}[\Psi]$ for $\tau\in\text{Out}(\mc R)
        $ non-trivial.
	\end{prop}
	\begin{proof}
		Let $\tau$ have order $2$. Since $\tau$ gives rise to a permutation of the simple roots $(e_i)_{1\leq i\leq r}$, the intersection of the kernels of the screening operators remains unchanged. Moreover, the stress-energy tensor $\Wb^{(2)}[\Psi]=\ps{Q,\partial^2\Psi}-\ps{\partial\Psi,\partial\Psi}$ is also invariant under $\Psi\to\Psi^*$. As a consequence, we see that the commutation relations $[\Wb^{(2)}_{n},\Wb^{(s_i)}_{m}]=\left(n(s_i-1)-m\right)\Wb^{(s_i)}_{m}$ are also preserved. Thus changing the currents from $\Wb^{(s_i)}[\Psi]$ to $\Wb^{(s_i)}[\Psi^*]$ preserves the $W$-algebra and the property of being a Virasoro primary field, whatever basis of currents is chosen. We denote such a basis by $\left(\bm w^{(s_i)}\right)_{1\leq i\leq r}$.

        Let us now construct the currents by showing inductively that for all $1\leq i\leq r$, there exist currents $\Wb^{(s_j)}$, $1\leq j\leq i$, such that $\MW{\g}^{(\leq s_i)}\coloneqq \bigoplus_{n\leq s_i}\MW{\g}^{(n)}$ is freely generated by acting on $\vac$ with the modes of these currents, and with $\Wb^{(s_j)}[\Psi^*]=\pm \Wb^{(s_j)}[\Psi]$. Here the grading is the one from the proof of~\cite[Theorem 4.3]{Cer_VOA}. The case of $i=1$ corresponds to the stress-energy tensor (of spin $s=2$), so we can focus on the inductive step. Let us assume that we have been able to define currents of spins $s_j$, $1\leq j\leq i-1$, satisfying these assumptions. Then, following the proof of~\cite[Theorem 4.6.9]{FF_QG},  the subspace $\MW{\g}^{(s_i)}$ of $\MW{\g}$ at level $s_i$ is given, since we have assumed that $\g\neq \mathfrak{so}_{4n}$, by
        \[
            \MW{\g}^{(s_i)}=\C \bm w^{(s_i)}_{-s_i}\vac\bigoplus \text{span}\left\{\Wb^{(s_1)}_{-\lambda_1}\cdots\Wb^{(s_{i-1})}_{-\lambda_{i-1}}\vac\right\}.
        \]
        Since $\tau$ preserves the space $\MW{\g}$, we can thus expand $\bm w_{-s_i}^{(s_i)}[\Psi^*]\vac = \lambda_1 \bm w_{-s_i}^{(s_i)}[\Psi]\vac + \Db^{<s_i}$ where $\Db^{<s_i}$ is generated using the modes of the currents $\Wb^{(s_j)}[\Psi]$ for $j<s_i$. By induction, we can further decompose $\Db^{<s_i}=\Db^{<s_i}_N+\Db^{<s_i}_D$ where $\tau \Db^{<s_i}_N=\Db^{<s_i}_N$ while $\tau \Db^{<s_i}_D=-\Db^{<s_i}_D$. 
        Now since $\tau$ is an involution, applying $\tau$ once again we obtain 
        $$
            \left(1+\lambda_1\right)\left(\bm w_{-s}^{(s)}[\Psi^*]-\bm w_{-s}^{(s)}[\Psi]\right)\vac = -2\Db^{<s_i}_D.
        $$
        Here we need to distinguish according to the value of $\lambda_1$. Let us first assume that $\lambda_1\neq-1$. We can then define $\Wb^{(s)}$ to be the vertex operator (see~\cite[Equation (3.19)]{Cer_VOA}) associated to $\bm w_{-s}^{(s)}[\Psi]\vac+\frac1{1+\lambda_1}\Db^{<s_i}_D$. Then by construction we have $\Wb^{(s_i)}[\Psi^*]=\Wb^{(s_i)}[\Psi]$. Moreover since $\bm w_{-s}^{(s)}[\Psi]\vac$ and $\frac1{1+\lambda_1}\Db^{<s_i}_N$ are linearly independent, the induction hypothesis is satisfied.  
        If we rather assume that $\lambda_1=-1$, then $\left(\bm w_{-s}^{(s)}[\Psi^*]+\bm w_{-s}^{(s)}[\Psi]\right)\vac = \Db^{<s_i}_D$, so that we rather define $\Wb^{(s)}$ to be the vertex operator associated to $\bm w_{-s}^{(s)}[\Psi]\vac-\frac12\Db^{<s_i}_D$. We then obtain $\Wb^{(s)}[\Psi^*]=-\Wb^{(s)}[\Psi]$ and the induction hypothesis is satisfied:
        we can always find currents that verify the two first points of the proposition, and such that $\Wb^{(s_i)}[\Psi^*]=\pm\Wb^{(s_i)}[\Psi]$. 
        
        It remains to see that this sign, in the $\sl_n$ case, is determined by the spin $s_i$. To this end, we first choose the basis of currents defined using a Miura transform~\cite{FaLu}, that we denote $\left(\bm w^{(s_i)}\right)_{1\leq i\leq r}$: we refer to~\cite[Subsection B.2]{Cer_VOA} for more details, and from which we borrow the notations. These currents admit an expansion of the form $\bm w^{(s_i)}=\sum_{j=0}^{s_i-1}q^j P_j(\Psi)$, where $q=\gamma+\frac2\gamma$ and $P_j(\Psi)$ is a $(s_i-j)$-multilinear form in the derivatives of $\Psi$. For instance
        \begin{equation}
                 P_0(\Psi)=\sum_{\substack{\mc P\subseteq\{1,\cdots,n\} \\ \norm{\mc P}=s_i}}\prod_{j\in\mc P}\ps{h_j,\partial\Psi}.
        \end{equation}
        For $\tau\neq\Id$, we have $\tau e_k=e_{n-k}$ so $\tau h_i=-h_{n+1-i}$ where $h_i=\omega_1-\sum_{k=1}^{i-1}e_k$ (with the notations from~\cite[Subsection B.2]{Cer_VOA}). This implies that $P_0(\Psi^*)=(-1)^{s_i}P_0(\Psi)$. 
        Now by contradiction assume that $s_i$ is odd and that $\Wb^{(s_i)}[\Psi^*]=+\Wb^{(s_i)}[\Psi]$. Then the equalities $P_0[\Psi^*]=-P_0[\Psi]$ and $\tau\Db^{<s_i}_D=-\Db^{<s_i}_D$ imply $P_0[\Psi]+\Db^{<s_i}_{D,0}[\Psi]=0$, where $\left(\Db^{<s_i}_D[\Psi]\right)_{-s_i}\vac=\Db^{<s_i}_D$ and $\Db^{<s_i}_{D,0}[\Psi]$ is the $s_i$-multilinear form in the expansion of $\Db^{<s_i}_D[\Psi]$. But by construction, in the $\sl_n$ case, $\Db^{<s_i}_{D,0}[\Psi]$ is of the form
        \[
            \Db^{<s_i}_{D,0}[\Psi]=\sum_{\bm p}c_{\bm p}\prod_{j=1}^{i-1}\left(\sum_{\substack{\mc P\subseteq\{1,\cdots,n\} \\ \norm{\mc P}=s_j}}\prod_{l\in\mc P}\ps{h_l,\partial\Psi}\right)^{p_j}
        \]
        where the sum ranges over $\bm p=(p_1,\cdots,p_{j-1})\in\N^r$ such that $\sum_{j=1}^{j-1}p_js_j=s_i$, and $c_{\bm p}$ are constants. The latter being linearly independent (for instance by viewing them as polynomials in the variables $X_i=\ps{h_i,\partial\Psi}$) from $P_0$, we cannot have $P_0[\Psi]+\Db^{<s_i}_{D,0}[\Psi]=0$. Hence we arrive to a contradiction: odd spin implies $\lambda=-1$. The same argument shows that if the spin is even we must have $\lambda=1$, and this concludes the proof. 
        We believe that a similar argument holds in the $\mathfrak{so}_{2n}$ case using the Miura transform~\cite{FaLuDn}.
	\end{proof}
	
	\subsubsection{Ward identities}
	An explicit manifestation of this symmetry is the existence of \textit{Ward identities}, that should arise when the $W$-currents are inserted within correlation functions. When $\Sigma=\S^2$, such a statement has been proved in~\cite{Toda_OPEWV} for $\g=\sl_3$ and for general $\g$ in~\cite{Cer_VOA} when a so-called \textit{neutrality condition} is assumed to hold. In the boundary case, with Neumann boundary conditions and with $\g=\sl_3$, such a statement was proved in~\cite{CH_sym1, CH_sym2}, revealing features unknown in the physics literature. We hope to explore in more depth general cases where the Cardy case is considered instead of the Neumann one, and when the underlying Lie algebra is not necessarily $\sl_3$. 
	\appendix
\section{Proofs of the main statements} \label{sec:appendix_proof} 
We gather here the proofs of the statements defining the probabilistic path integral and the correlation functions. The case of a closed Riemann surface being treated in the exact same fashion as the open one we will focus on the latter. Likewise the partition function being a special case of the correlation functions our goal is to prove Theorems~\ref{thm:existence} and~\ref{thm:axioms}. Hereafter we denote $\bm x=(x_1,\cdots,x_N,s_1,\cdots,s_M)$ and $\bm \alpha=(\alpha_1,\cdots,\alpha_N,\frac12\beta_1,\cdots,\frac12\beta_M)$.

A key ingredient is the following formulation of the Girsanov (or Cameron-Martin) theorem:
    \begin{thm} \label{thm:girsanov}
    	Let $(\Sigma,g)$ be a Riemannian surface, $(X(x))_{x\in \Sigma} = (X_1(x),...,X_n(x))_{x\in \Sigma}$ be a family of smooth centered Gaussian fields on $\Sigma$ and $Y$ be a random variable in the $L^2$-closure of ${\rm Span}(X(x))_{x\in \Sigma}$. For any $F$ bounded over the space of continuous functions, 
    	$$
    	\E\left[e^{Y-\frac12 \E[Y^2]}F(X(x))_{x\in \Sigma}\right] = \E\left[F(X(x) + \E[YX(x)])_{x\in \Sigma}\right].
    	$$
    \end{thm}

\subsection{Existence of correlation functions (Theorem~\ref{thm:existence})}We start by establishing the Seiberg bounds~\eqref{eq:Seiberg}. In view of the Weyl anomaly formula that we will prove next we can always assume that $g=g_0$ is a uniform metric of type 1. Without loss of generality and since, for $z\in\C$, $\norm{e^z}=e^{\Re(z)}$ we also assume the cosmological constants to be real-valued.

Recall the decomposition~\eqref{eq:decomposition Cardy GFF} of the field $\X^{g,C}$ in terms of $\X^{N,g}$ and $\X^{D,g}$ two independent vectorial Neumann (resp. Dirichlet) GFFs. Then by Girsanov's theorem~\ref{thm:girsanov} for any $F$ bounded continuous over $H^{-1}(\Sigma\to\a,g)$ we have 
\begin{align*}
    &\expect{\prod_{k=1}^NV_{\alpha_k,\eps}(x_k)\prod_{l=1}^MV_{\beta_l,\eps}(s_l)F\left(\X^{g,C}_\eps+\c\right)}=C_{\eps,\bm{x},\bm{\alpha}}\expect{F\left(\X^{g,C}_\eps+\c+H_{\eps,\bm x,\bm\alpha}\right)}\qt{with}\\
    &C_{\eps,\bm x,\bm\alpha}=\prod_{k=1}^{N}e^{\frac{\norm{\alpha_k}^2}{2}W_{g,\eps}(x_k)}\prod_{l=1}^{M}e^{\frac{\norm{\beta_l}^2}{4}W_{g,\eps}(x_k)}\prod_{1\leq k<l\leq N+M}e^{\ps{\alpha_k,\alpha_l}G_{g,\eps}(x_k,x_l)}\qt{and}\\
    &H_{\eps,\bm x,\bm\alpha}=\sum_{k=1}^{N+M}\alpha_k G_{g,\eps,\eps}(\cdot,x_k).
\end{align*}
Here $G_{g,\eps}(x,y)=\mu_{g,\eps}[G_g(\cdot,y)](x)$, $G_{g,\eps,\eps}(x,y)=\mu_{g,\eps}[z\mapsto \mu_{g,\eps}[G_g(\cdot,z)](x)](y)$, while $W_{g,\eps}(x_k)=W_{g_0}(x_k)+o(1)$ as $\eps\to0$ (see Equation~\eqref{eq:log divergence diagonal}). Hence $C_{\eps,\bm x,\bm\alpha}=C_{0,\bm x,\bm \alpha}(1+o(1))$. As a consequence the regularized correlation functions are
	\begin{align*}
		&A_\eps = C_{0,\bm{x},\boldsymbol{\alpha}} \int_{\a_N} e^{\ps{\sum \alpha_k + \frac12 \sum \beta_l - Q\chi(\Sigma),\c}} \E\left[e^{-\sum_{i=1}^r \mu_{B,i} e^{\ps{\gamma e_i,\c}} Z_{i,\eps} -  \sum_{j=1}^{d_N}e^{\ps{\frac{\gamma}{2}f_j,\c}} Z_{j,\eps}^{\partial}}\right]\d\c (1+o(1)),\\
		&Z_{i,\eps} = \int_{\Sigma} \eps^{\frac{\gamma_i^2}{4}} e^{\ps{\gamma e_i, \X^{C,g}_\eps(x)+H_{\eps,\bm x,\bm\alpha}(x)}} \d v_g(x),\quad Z_{j,\eps}^{\partial} = \int_{\partial\Sigma}  \eps^{\frac{\norm{\gamma f_j}^2}{2}} e^{\ps{\frac{\gamma}{2}f_j, X_\eps^{N,g}(s)+H_{\eps,\bm x,\bm\alpha}(s)}} \mu_j(\d\lambda_g)(s).
    \end{align*}
    To see that the Seiberg bounds are necessary conditions, we bound for any positive $M>0$ 
    \begin{align*}
		A_\eps \ge C_{\eps,\bm{x},\boldsymbol{\alpha}} \int_{\a_N} e^{\ps{\sum \alpha_k + \frac12 \sum \beta_l - Q\chi(\Sigma),\c}}& e^{-\sum_{i=1}^r \left(\mu_{B,i} e^{\ps{\gamma e_i,\c}} + e^{\ps{\frac{\gamma}{2}f_j,\c}} \right)M}\\
        &\times \P(\forall i,j \ Z_{i,\eps}\leq M\text{ and } Z_{j,\eps}^{\partial} \leq M) \d\c (1+o(1)).
	\end{align*}
    The latter probability being non-zero we see that $A_\eps$ is infinite as soon as there is a $i\in\{1,\cdots,r\}$ such that $\ps{\sum \alpha_k + \frac12 \sum \beta_l - Q\chi(\Sigma),\omega_i^\vee} \le 0$: the first condition of the Seiberg bounds must be satisfied. Besides, if f $\ps{\alpha_k-Q,e_i}\geq 0$ for some $1\leq k\leq N$ and $1\leq i\leq r$ we have that for $\eps$ small enough and for any $s<0$, $\int_{B(z_k,\eps)}\norm{x-z_k}^{-\ps{\alpha_k,\gamma e_i}}e^{\ps{\gamma e_i,\X^{g,C}}}\d v_g=+\infty$ almost surely (and likewise for boundary insertions) ; see for instance the paragraph between Equations (3.9) and (3.10) in~\cite{HRV18}. This entails that the second condition in the Seiberg bounds has to be verified too.

    We now assume the conditions~\eqref{eq:Seiberg} to hold and show that they are sufficient for the regularized partition function to converge and to be non-trivial. We actually prove a stronger statement that may be used to extend the Seiberg bounds. Let us introduce the shortcuts $\mu_{i,\eps}^N(x)$ and $\mu_{i,\eps}^D(x)$ for the regularized Neumann and Dirichlet GMC densities, respectively. 

    \begin{lemma}
        Assume that $\ps{\alpha-Q,e_i}<0$. Then for $x\in\Sigma$ and $\rho>0$ such that $B_g(x,2\rho)\cap\partial\Sigma=\emptyset$,
        $$
        \sup\limits_{\eps>0}\E \left[ \left(\int_{B_g(x,\rho)}  e^{\ps{\gamma e_i,\alpha} G_{g,\eps}(x,y)} \mu_{i,\eps}^N(y) \mu_{i,\eps}^D(y) \d v_g(y)\right)^p\right] < +\infty
        $$
        if and only if $p\in(-\infty,\frac{4}{\gamma_i^2}\wedge\frac1\gamma \ps{Q-\alpha,e_i^\vee})$. 
    \end{lemma}
    \begin{proof}
        By Kahane's convexity inequality this result follows from~\cite{DKRV16}[Lemma A.1]; see the proof of the next lemma.
    \end{proof}

    \begin{lemma}
        Suppose that $\ps{\beta-Q,e_i}<0$ and $s\in\partial\Sigma$. Then for all $1>\delta>0$,
        $$
        \sup\limits_{\eps>0}\E \left[ \left(\int_{B_g(s,1-\delta)\cap\Sigma}  e^{\frac12 \ps{\gamma e_i,\beta} G_{g,\eps}(s,x)} \mu_{i,\eps}^N(x) \mu_{i,\eps}^D(x) \d v_g(x)\right)^p\right] < +\infty
        $$
        if $p\in(-\infty,\frac{2}{\gamma_i^2}\wedge\frac1\gamma \ps{Q-\beta,e_i^\vee})$.
    \end{lemma}

    \begin{proof}
        Suppose first that $p>0$. We can get rid of the Dirichlet part of the Cardy GMC measure by using Kahane's convexity inequality (see \cite[Corollary 25]{Huang_moments}). Indeed, a simple computation shows that
        $$
        \E[\ps{e_i,X^{g,C}_{\eps}(x)}\ps{e_i,X^{g,C}_{\eps}(y)}] = |e_i|^2G_{g,\eps,\eps}^N(x,y)-2\ps{e_i,p_D e_i}G_{g,\eps,\eps}(x,\sigma(y))
        $$
        with $\ps{e_i,p_De_i}\ge0$, so that we can compare
        $$
        \E[\ps{e_i,X^{g,C}_{\eps}(x)}\ps{e_i,X^{g,C}_{\eps}(y)}] \le \E[\ps{e_i,X^{g,N}_{\eps}(x)}\ps{e_i,X^{g,N}_{\eps}(y)}] + C
        $$
        for $x,y \in B_g(s,1-\delta)\cup\Sigma$. Since the function $t\mapsto t^p$ is convex for $t\ge0$, Kahane's inequality implies that the expectation in the statement of the lemma can be upper-bounded by (a constant times) the same integral with only Neumann type density. Moreover, thanks to Lemma~\ref{lemma:estimates_green_function}, we can work on the Poincaré disk $(B(0,1-\delta),g_P)$. Thus, the result follows from~\cite[Lemmas 6.3 and 6.9]{HRV18}. \\
        When $p<0$ the situation is simpler since the integral in the statement of the lemma can be lower-bounded by the same integral on $B_g(s,1-\delta)\cap \{x\in\Sigma : d_g(x,\partial\Sigma)>\eta\}$ with $\eta>0$ small, so the situation is the same as in the bulk.
    \end{proof}

\subsection{Weyl anomaly and diffeomorphism invariance (Theorem~\ref{thm:axioms})} 
        First, notice that it is enough to show the result for $g=g_0$ the uniform type 1 metric. Indeed, suppose that $g'=e^\varphi g$ are two metrics as in Theorem~\ref{thm:axioms}, there exist two smooth functions $\psi$ and $\xi$ such that $g'=e^{\psi}g_0$ and $g_0=e^\xi g$, with $\varphi=\psi+\xi$. Then we have using \eqref{eq:conf change curvature}:
        \begin{align*}
            &\int_{\Sigma} (|d_{g_0}\psi|_{g_0}^2 + 2R_{g_0}\psi) \d v_{g_0} + 4\int_{\partial\Sigma} k_{g_0}\psi \d\lambda_{g_0} + \int_{\Sigma} (|d_{g}\xi|_{g}^2 + 2R_{g}\xi) \d v_{g} + 4\int_{\partial\Sigma} k_{g}\xi \d\lambda_{g} \\
            &= \int_{\Sigma} (|d_{g}\varphi|_{g}^2 + 2R_{g}\varphi - 2\ps{d_g\psi,d_g\xi}_g - 2\psi\Delta_g\xi) \d v_{g} + \int_{\partial\Sigma} (4k_g\varphi + 2\psi \partial_{\vec{n}_g} \xi )\d\lambda_g\\
            &= \int_{\Sigma} (|d_{g}\varphi|_{g}^2 + 2R_{g}\varphi) \d v_{g}+ \int_{\partial\Sigma} 4k_g\varphi\d\lambda_g
        \end{align*}
        by integration by parts. This shows that we can restrict ourselves to the case of $g=e^\varphi g_0$ with $g_0$ the uniform type 1 metric. We choose to present the proof in the most general case where the field under consideration is the Cardy GFF $\X^{g,C}$. Remember that in this case the $\mu_j$'s associated with the short roots are set to $0$. One naturally recovers the Neumann case by setting $d_N=r$. We omit the $\eps$-regularization in the following since the convergence of the correlation functions for a general smooth metric $g$ is justified by the computations to come. We first use Lemmas~\ref{lemma:change of metric GFF} and \ref{lemma:conformal change GMC cardy} and make a change of variable $\c \to \c + m_g(\X^{g_0,C})$ (which is possible since the Dirichlet part of $m_g(\X^{g_0,C})$ is $0$, so that $m_g(\X^{g_0,C})$ lives in $\a_N$) and apply the Gauss-Bonnet theorem to restrict ourselves to
    	\begin{align} \label{eq: proof_weyl part funct}
            &\int_{\a_N} e^{-\ps{Q,\c}\chi(\Sigma)} \E\left[ F(\X^{g_0,C} + \c) \exp\left( -\frac{1}{4\pi} \int_{\Sigma} R_{g}\ps{Q,\X^{g_0,C}} \d v_{g} -\frac{1}{2\pi} \int_{\partial\Sigma} k_{g}\ps{Q,\X^{g_0,C}} \d \lambda_{g} \right.\right.\nonumber\\
            &\left.\left.- \sum_{i=1}^r \mu_{B,i} e^{\ps{\gamma e_i,\c}} \int_{\Sigma} e^{(1+\frac{\gamma_i^2}{4})\varphi} e^{\ps{\gamma e_i,\X^{g_0,C}}} \d v_{g_0} - \sum_{j=1}^{d_N} \mu_{j} e^{\ps{\frac{\gamma f_j}{2},\c}} \int_{\partial\Sigma} e^{(1+\frac{\gamma^2|f_j|^2}{4})\frac{\varphi}{2}}e^{\ps{\frac{\gamma f_j}{2},\X^{g_0,C}}} \d \lambda_{g_0}\right)\right] \d\c. 
    	\end{align}
    	Now we would like to apply the Girsanov transform to the term $e^Y$, where (using~\eqref{eq:conf change curvature})
        \begin{align*}
            Y&\coloneqq -\frac{1}{4\pi} \int_{\Sigma} R_{g}\ps{Q,\X^{g_0}} \d v_{g}-\frac{1}{2\pi} \int_{\partial\Sigma} k_{g}\ps{Q,\X^{g_0,C}} \d \lambda_{g}\\
            &=\frac1{4\pi} \left( \int_{\Sigma} \Delta_{g_0}\varphi\ps{Q,\X^{g_0,C}} \d v_{g_0}- \int_{\partial\Sigma} \partial_{\vec{n}_{g_0}}\varphi \ps{Q,\X^{g_0,C}} \d \lambda_{g_0}\right).
        \end{align*} 
        Hence we shall compute the variance of $Y$. For this we use the fact that $Q\in\a_N$, together with \eqref{eq:gauss-green formula neumann} and integration by parts to infer that
    	\begin{align*}
    		\E[Y^2] &= \frac{|Q|^2}{16\pi^2} \left( \int_\Sigma \Delta_{g_0} \varphi(x)\left( \int_{\Sigma}\Delta_{g_0} \varphi(y) G^N_{g_0}(x,y) \d v_{g_0}(y) - \int_{\partial\Sigma} \partial_{\vec{n}_{g_0}}\varphi(y)G^N_{g_0}(x,y)\d\lambda_{g_0}(y)\right)\d v_{g_0}(x) \right.\\
            &\left. +\int_{\partial\Sigma} \partial_{\vec{n}_{g_0}}\varphi(y)\left( \int_{\partial\Sigma}\partial_{\vec{n}_{g_0}}\varphi(x)G^N_{g_0}(x,y) \d \lambda_{g_0}(x) - \int_{\Sigma} \Delta_{g_0} \varphi(x) G^N_{g_0}(x,y) \d v_{g_0}(x)\right)\d \lambda_{g_0}(y)\right)\\
            &= \frac{|Q|^2}{8\pi} \left(m_{g_0}\varphi \left(\int_\Sigma \Delta_{g_0}\varphi \d v_{g_0}-\int_{\partial\Sigma} \partial_{\vec{n}_{g_0}}\varphi \d\lambda_{g_0}\right) + \int_\Sigma \varphi\Delta_{g_0}\varphi \d v_{g_0}-\int_{\partial\Sigma} \varphi\partial_{\vec{n}_{g_0}}\varphi \d\lambda_{g_0} \right)\\
            &= \frac{|Q|^2}{8\pi} \int_{\Sigma} |d_{g_0}\varphi|_{g_0}^2 \d v_{g_0}.
    	\end{align*}
    	In the same way, we obtain 
        \begin{align*}
            \E[Y\ps{u,\X^{g_0,C}}] &= \frac{\ps{Q,p_Nu}}{4\pi} \left(\int_\Sigma \Delta_{g_0}\varphi(x) G_{g_0}^N(x,\cdot)\d v_{g_0}(x) - \int_{\partial\Sigma} \partial_{\vec{n}_{g_0}}\varphi(x) G_{g_0}^N(x,\cdot)\d\lambda_{g_0}(x) \right)\\
            &= -\frac{\ps{Q,p_N u}}{4\pi}2\pi(\varphi-m_{g_0}(\varphi))= -\frac{\ps{Q,u}}{2}(\varphi-m_{g_0}(\varphi))
        \end{align*}
    	for any $u \in \a$. Notice that the above Girsanov shift coincides with the conformal anomaly of the GMC measures in \eqref{eq: proof_weyl part funct}, since the short roots are not taken into account in the boundary potential. Hence applying Girsanov theorem allows us to rewrite \eqref{eq: proof_weyl part funct} as
    	\begin{equation*}
    		\begin{split}
    			&\int_{\a_N} e^{\frac{|Q|^2}{16\pi}\int_{\Sigma} |d_{g_0}\varphi|_{g_0}^2 \d v_{g_0} - \ps{Q,\c}\chi(\Sigma)} \E\left[F(\c+\X^{g_0,C} - \frac{Q}{2} (\varphi - m_{g_0}(\varphi))) \right. 
    			\\&\left.\times\exp\left(-\sum_{i=1}^r \mu_{B,i} e^{\ps{\gamma e_i,\c+\frac{Q}{2}m_{g_0}(\varphi)}} \int_{\Sigma} e^{\ps{\gamma e_i,\X^{g_0,C}}} \d v_{g_0}-\sum_{i=1}^{d_N} \mu_{i} e^{\ps{\frac{\gamma f_j}2,\c+\frac{Q}{2}m_{g_0}(\varphi)}} \int_{\partial\Sigma} e^{\ps{\frac{\gamma f_j}2,\X^{g_0,C}}} \d \lambda_{g_0}\right)\right] \d\c.
    		\end{split}
    	\end{equation*}
    	Now we make the change of variables in the zero-mode $\c \to \c -\frac Q2 m_{g_0}(\varphi)$ (which is possible since $Q\in\a_N$) to obtain that $\eqref{eq: proof_weyl part funct}$ equals
    	\begin{equation*}
    		e^{\frac{|Q|^2}{16\pi}\int_{\Sigma} |d_{g_0}\varphi|_{g_0}^2\d v_{g_0} + \frac{|Q|^2}{2}\chi(\Sigma)m_{g_0}(\varphi)} \mc{Z}_C(\Sigma,g_0)^{-1} \ps{F(\cdot-\frac{Q}2\varphi)}_{g_0}. 
    	\end{equation*}
    	We conclude with the Gauss-Bonnet theorem (Equation~\eqref{eq:gauss-bonnet}) and Proposition~\ref{prop:conformal anomaly boundary}. Finally, the fact that the partition function is diffeomorphism invariant directly follows from the fact that all the objects involved are so. Indeed, if $\psi : \Sigma \to \Sigma$ is an orientation-preserving diffeomorphism preserving $\partial\Sigma$, then for any smooth Riemannian metric $g$ we have $R_{\psi^*g} = R_g\circ\psi$ and $k_{\psi^*g} = k_g\circ\psi$, while $\X^{\psi^*g,C} \eqlaw \X^{g,C}\circ\psi$ since we have the following diffeomorphism invariance satisfied by the Green's functions:
        $$
        G_{\psi^*g}^N(x,y) = G_g^N(\psi(x),\psi(y)) \quad G_{\psi^*g}^D(x,y) = G_g^D(\psi(x),\psi(y)).
        $$
	
	\bibliography{main.bib}

@incollection{FaLuDn,
  title={Exactly soluble models of conformal quantum field theory associated with the simple Lie algebra Dn},
  author={Luk'yanov, Sergei L and Fateev, Vladimir A},
  booktitle={W-Symmetry},
  pages={325--332},
  year={1995},
  publisher={World Scientific}
}

@article{Huang_moments,
      title={Moment bounds for Gaussian multiplicative chaos with higher-dimensional singularities}, 
      author={Yichao Huang},
      year={2023},
      eprint={2302.06097},
      archivePrefix={arXiv},
      primaryClass={math.PR},
      url={https://arxiv.org/abs/2302.06097}, 
}

@article{CH_sym1,
	author = {Cerclé, B. and Huguenin, N.},
	year = {2024},
	JOURNAL = {Preprint, \href{http://arxiv.org/abs/2412.13874}{\textup{\nolinkurl{arXiv:2412.13874}}}},
	title = "{Higher-spin symmetry in the $\mathfrak{sl}_3$ boundary Toda conformal field theory I: Ward identities}",
}

@article{CH_sym2,
	author = {Cerclé, B. and Huguenin, N.},
	year = {2025},
	JOURNAL = {Preprint, \href{http://arxiv.org/abs/2503.20548}{\textup{\nolinkurl{arXiv:2503.20548}}}},
	title = "{Higher-spin symmetry in the $\mathfrak{sl}_3$ boundary Toda conformal field theory II: Singular vectors and BPZ equations}",
}

@article{weyl_law_ivrii,
  title={100 years of Weyl’s law},
  author={Ivrii, Victor},
  journal={Bulletin of Mathematical Sciences},
  volume={6},
  number={3},
  pages={379--452},
  year={2016},
  publisher={Springer}
}

@article{McKeanSinger,
    author = "McKean, H. P. and Singer, I. M.",
    title = "{Curvature and eigenvalues of the Laplacian}",
    journal = "J. Diff. Geom.",
    volume = "1",
    pages = "43--69",
    year = "1967"
}

@article{polyakov_bosonic_strings,
	title={Quantum geometry of bosonic strings},
	author={Polyakov, Alexander M},
	journal={Physics Letters B},
	volume={103},
	number={3},
	pages={207--210},
	year={1981},
	publisher={Elsevier}
}

@article{polyakov_fermionic_strings,
	title={Quantum geometry of fermionic strings},
	author={Polyakov, Alexander M},
	journal={Physics Letters B},
	volume={103},
	number={3},
	pages={211--213},
	year={1981},
	publisher={Elsevier}
}

@article{remy_annulus,
    author = "Remy, G.",
    title = "{Liouville quantum gravity on the annulus}",
    eprint = "1711.06547",
    archivePrefix = "arXiv",
    primaryClass = "math-ph",
    doi = "10.1063/1.5030409",
    journal = "J. Math. Phys.",
    volume = "59",
    number = "8",
    pages = "082303",
    year = "2018"
}

@article{PT02,
	title="{Boundary Liouville field theory: boundary three-point function}",
	author={B. Ponsot and J. Teschner},
	journal={Nuclear Physics B},
	year={2002},
	volume={622(1)},
	pages={309-327}
}

@article{Ang_zipper,
author = {Ang, M.},
title = {{Liouville conformal field theory and the quantum zipper}},
volume = {53},
journal = {The Annals of Probability},
number = {6},
publisher = {Institute of Mathematical Statistics},
pages = {2054 -- 2098},
year = {2025},
doi = {10.1214/24-AOP1726},
URL = {https://doi.org/10.1214/24-AOP1726}
}

@article{GRV19,
	author         = "Guillarmou, C. and Rhodes, R. and Vargas, V.",
	title          = "{Polyakov's formulation of $2d$ bosonic string theory}",
	year           = "2019",
	journal         = "Publications Math\'ematiques de l'IH\'ES",
	volume  = "130",
	pages = "111-185",
}

@article{RV14,
	author = {Rhodes, R. and Vargas, V.},
	title = {{Gaussian multiplicative chaos and applications: A review}},
	volume = {11},
	journal = {Probability Surveys},
	publisher = {Institute of Mathematical Statistics and Bernoulli Society},
	pages = {315 -- 392},
	year = {2014},
	doi = {10.1214/13-PS218},
	URL = {https://doi.org/10.1214/13-PS218}
}

@article{DKRV16,
	author    = "F. David and A. Kupiainen and R. Rhodes and V. Vargas",
	title     = "Liouville {Q}uantum {G}ravity on the {R}iemann {S}phere",
	journal   = "Communications in Mathematical Physics",
	year      = "2016",
	volume = " 342",
	pages ="869-907",
}

@article{Hos,
	author = "Hosomichi, K.",
	title = "{Bulk boundary propagator in Liouville theory on a disc}",
	journal = "JHEP",
	volume = "11",
	pages = "044",
	year = "2001"
}

@article{FaLi1,
	author = "Fateev, V.A. and Litvinov, A.V.",
	title = "{Correlation functions in conformal Toda field theory. I.}",
	eprint = "0709.3806",
	archivePrefix = "arXiv",
	primaryClass = "hep-th",
	reportNumber = "RUNHETC-2007-16, PTA-07-42",
	doi = "10.1088/1126-6708/2007/11/002",
	journal = "JHEP",
	volume = "11",
	pages = "002",
	year = "2007"
}

@article{FaRi,
	author = "Fateev, Vladimir and Ribault, Sylvain",
	title = "{Conformal Toda theory with a boundary}",
	eprint = "1007.1293",
	archivePrefix = "arXiv",
	primaryClass = "hep-th",
	reportNumber = "LPTA:10-054",
	doi = "10.1007/JHEP12(2010)089",
	journal = "JHEP",
	volume = "12",
	pages = "089",
	year = "2010"
}

@article{Ber17,
	author = {Berestycki, N.},
	title = {{An elementary approach to Gaussian multiplicative chaos}},
	volume = {22},
	journal = {Electronic Communications in Probability},
	publisher = {Institute of Mathematical Statistics and Bernoulli Society},
	pages = {1 -- 12},
	year = {2017},
	doi = {10.1214/17-ECP58},
	URL = {https://doi.org/10.1214/17-ECP58}
}

@article{FF_QG,
	author = "Feigin, B. and Frenkel, E.",
	editor = "Francaviglia, M. and Greco, S.",
	title = "{Integrals of motion and quantum groups}",
	eprint = "hep-th/9310022",
	archivePrefix = "arXiv",
	reportNumber = "YITP-K-1036",
	doi = "10.1007/BFb0094794",
	journal = "Lect. Notes Math.",
	volume = "1620",
	pages = "349--418",
	year = "1996",
}

@article{Cer_VOA,
	author = {Cerclé, B.},
	year = {2024},
	JOURNAL = {To appear in Proceedings of the London Mathematical Society},
	title = "{W-algebras, Gaussian Free Fields and $\mathfrak g$-Dotsenko-Fateev integrals}",
}

@article{CRV21,
	author    = "Cercl\'e, B. and Rhodes, R. and Vargas, V.",
	title     = "{Probabilistic construction of Toda conformal field theories}",
	journal = "Annales Henri Lebesgue",
	volume = "6",
	pages = "31-64",
	year      = "2023",
}

@article{Toda_OPEWV,
	author    = "Cercl\'e, B. and Huang, Y.",
	title     = "{Ward identities in the $\mathfrak{sl}_3$ Toda conformal field theory}",
	journal = "Communications in Mathematical Physics",
	volume = "393",
	pages = "419-475",
	year      = "2022",
}

@article{Toda_correl1,
	author    = "Cercl\'e, B.",
	title     = "{Three-point correlation functions in the $\mathfrak{sl}_3$ Toda theory I: Reflection coefficients}",
	JOURNAL = {Probability Theory and Related Fields},
	volume = {188},
	pages = {89–158},
	year      = "2024",
}

@article{Toda_correl2,
	author    = "Cercl\'e, B. ",
	title     = "{Three-point correlation functions in the $\mathfrak{sl}_3$ Toda theory II: the formula}",
	JOURNAL = {Journal of the European Mathematical Society},
	year      = "to appear",
}

@article{DRV16,
	author = {David,F.  and Rhodes,R.  and Vargas,V. },
	title = {Liouville quantum gravity on complex tori},
	journal = {Journal of Mathematical Physics},
	volume = {57},
	number = {2},
	pages = {022302},
	year = {2016},
	doi = {10.1063/1.4938107},
	URL = {https://doi.org/10.1063/1.4938107},
	eprint = {https://doi.org/10.1063/1.4938107}
}

@article{HRV18,
	author = "Huang, Y. and Rhodes, R. and Vargas, V.",
	doi = "10.1214/17-AIHP852",
	fjournal = "Annales de l'Institut Henri Poincar\'e, Probabilit\'es et Statistiques",
	journal = "Ann. Inst. H. Poincar\'e Probab. Statist.",
	month = "08",
	number = "3",
	pages = "1694--1730",
	publisher = "Institut Henri Poincar\'e",
	title = "Liouville quantum gravity on the unit disk",
	url = "https://doi.org/10.1214/17-AIHP852",
	volume = "54",
	year = "2018"
}

@article{OPS88,
	title = "{Extremals of determinants of Laplacians}",
	journal = "Journal of Functional Analysis",
	volume = "80",
	number = "1",
	pages = "148 - 211",
	year = "1988",
	issn = "0022-1236",
	doi = "https://doi.org/10.1016/0022-1236(88)90070-5",
	url = "http://www.sciencedirect.com/science/article/pii/0022123688900705",
	author = "B Osgood and R Phillips and P Sarnak"
}

@article{KRV_DOZZ,
	ISSN = {0003486X, 19398980},
	URL = {https://www.jstor.org/stable/10.4007/annals.2020.191.1.2},
	author = {A. Kupiainen and R. Rhodes and V. Vargas},
	journal = {Annals of Mathematics},
	number = {1},
	pages = {81--166},
	publisher = {[Annals of Mathematics, Trustees of Princeton University on Behalf of the Annals of Mathematics, Mathematics Department, Princeton University]},
	title =" {Integrability of Liouville theory: proof of the DOZZ formula}",
	volume = {191},
	year = {2020}
}

@ARTICLE{Za85,
	author = {{Zamolodchikov}, A.~B.},
	title = "{Infinite additional symmetries in two-dimensional conformal quantum field theory}",
	journal = {Theoretical and Mathematical Physics},
	year = 1985,
	month = dec,
	volume = {65},
	number = {3},
	pages = {1205-1213},
	doi = {10.1007/BF01036128},
	adsurl = {https://ui.adsabs.harvard.edu/abs/1985TMP....65.1205Z}
}

@article{FaLu,
	author = "Fateev, V.A. and Lukyanov, S. L.",
	title = "{The Models of Two-Dimensional Conformal Quantum Field Theory with Z(n) Symmetry}",
	reportNumber = "IC-87-238",
	doi = "10.1142/S0217751X88000205",
	journal = "Int. J. Mod. Phys. A",
	volume = "3",
	pages = "507",
	year = "1988"
}

@article{GKRV,
	author    = "C. Guillarmou and A. Kupiainen and R. Rhodes and V. Vargas",
	title     = "{Conformal bootstrap in Liouville Theory}",
	journal = "To appear in Acta Mathematica",
	year      = "2020",
}

@article{GKRV_Segal,
	author    = "C. Guillarmou and A. Kupiainen and R. Rhodes and V. Vargas",
	title     = "{Segal's axioms and bootstrap for Liouville theory}",
	JOURNAL = {Annals of Mathematics},
	year      = "to appear",
}

@article{ARSZ,
	author    = "Ang, M. and Remy, G. and Sun, X. and Zhu, T.",
	title     = "{Derivation of all structure constants for boundary Liouville CFT}",
	JOURNAL = {Preprint, \href{http://arxiv.org/abs/2305.18266}{\textup{\nolinkurl{arXiv:2305.18266}}}},
	year      = "2023",
}

@incollection{Seg04,
	author="Segal, G.",
	title="The definition of conformal field theory",
	bookTitle="Topology, Geometry, and Quantum Field Theory. Proc. Oxford 2002",
	year="2004",
	publisher="Oxford Univ. Press 2004",
}

@article{Pol81,
	author    = "A. Polyakov",
	title     = "{Quantum Geometry of bosonic strings}",
	journal   = "Physics Letters B",
	volume = "103",
	pages = "207:210",
	year      = "1981",
}

@article{ZZ96,
	title = "{Conformal bootstrap in Liouville field theory}",
	journal = "{Nuclear Physics B}",
	volume = {477},
	number = {2},
	pages = {577-605},
	year = {1996},
	issn = {0550-3213},
	doi = {https://doi.org/10.1016/0550-3213(96)00351-3},
	url = {https://www.sciencedirect.com/science/article/pii/0550321396003513},
	author = {Zamolodchikov, A. and Zamolodchikov, Al.}
}

@article{BPZ,
	title = "Infinite conformal symmetry in two-dimensional quantum field theory",
	journal = "Nuclear Physics B",
	volume = "241",
	number = "2",
	pages = "333 - 380",
	year = "1984",
	issn = "0550-3213",
	doi = "https://doi.org/10.1016/0550-3213(84)90052-X",
	url = "http://www.sciencedirect.com/science/article/pii/055032138490052X",
	author = "A.A. Belavin and A.M. Polyakov and A.B. Zamolodchikov"
}

@article{DO94,
	title = "{Two- and three-point functions in Liouville theory}",
	journal = "Nuclear Physics B",
	volume = "429",
	number = "2",
	pages = "375 - 388",
	year = "1994",
	issn = "0550-3213",
	doi = "https://doi.org/10.1016/0550-3213(94)00352-1",
	url = "http://www.sciencedirect.com/science/article/pii/0550321394003521",
	author = "H. Dorn and H.-J. Otto"
}

@article{DS11,
	author    = "B. Duplantier and S. Sheffield",
	title     = "{Liouville Quantum Gravity and KPZ}",
	journal   = "Inventiones mathematicae",
	volume= "185",
	pages="333",
	year      = "2011",
}

@book{PW21,
	AUTHOR = {Powell, E. and Werner, W.},
	TITLE = {Lecture Notes on the Gaussian Free Field},
	SERIES = {Cours Spécialisés},
	PUBLISHER = {SMF},
	YEAR = {2021},
}

@book{Hum72,
	author    = "Humphreys, J. E.",
	title     = "{Introduction to Lie Algebras and Representation Theory}",
	publisher = "Springer, Berlin",
	year      = "1972",
}

@article{Alv83,
    author = "Alvarez, O.",
    title = "{Theory of Strings with Boundaries: Fluctuations, Topology, and Quantum Geometry}",
    reportNumber = "CLNS 82/539",
    doi = "10.1016/0550-3213(83)90490-X",
    journal = "Nucl. Phys. B",
    volume = "216",
    pages = "125--184",
    year = "1983"
}

@article{RV10,
	author = "Robert, R. and Vargas, V.",
	doi = "10.1214/09-AOP490",
	fjournal = "The Annals of Probability",
	journal = "Ann. Probab.",
	month = "03",
	number = "2",
	pages = "605--631",
	publisher = "The Institute of Mathematical Statistics",
	title = "Gaussian multiplicative chaos revisited",
	url = "https://doi.org/10.1214/09-AOP490",
	volume = "38",
	year = "2010"
}

@article{FZZ,
	author    = "V. Fateev and A. Zamolodchikov and Al. Zamolodchikov",
	title     = "{Boundary Liouville Field Theory I. Boundary State and Boundary Two-point Function}",
	JOURNAL = {Preprint, \href{http://arxiv.org/abs/0001012}{\textup{\nolinkurl{arXiv:0001012}}}},
	YEAR = {2000},
}

@article{Wu,
	author    = "B. Wu",
	title     = "{Liouville conformal field theory on Riemann surface with boundaries}",
	JOURNAL = {Preprint, \href{http://arxiv.org/abs/2203.11721}{\textup{\nolinkurl{arXiv:2203.11721}}}},
	YEAR = {2022},
}

@article {remy1,
	AUTHOR = {Remy, G.},
	TITLE = {The {F}yodorov-{B}ouchaud formula and {L}iouville conformal
	field theory},
	JOURNAL = {Duke Math. J.},
	FJOURNAL = {Duke Mathematical Journal},
	VOLUME = {169},
	YEAR = {2020},
	NUMBER = {1},
	PAGES = {177--211},
	ISSN = {0012-7094},
	MRCLASS = {60G15 (60G57 60G60 81T08 81T40)},
	MRNUMBER = {4047550},
	DOI = {10.1215/00127094-2019-0045},
	URL = {https://doi.org/10.1215/00127094-2019-0045},
}
	\bibliographystyle{plain}
	
\end{document}